\documentclass[a4paper,10pt]{article}
\usepackage{amsmath,amssymb,amsthm,mathtools,xfrac,mathrsfs,stmaryrd} 
\usepackage{float} 
\usepackage{color}
\usepackage{graphics}
\usepackage{enumerate}
\usepackage[dvipsnames]{xcolor}
\usepackage[hyperindex]{hyperref}
\usepackage{tikz}
\usetikzlibrary{patterns}

\newtheorem{theorem}{Theorem}[section]

\newtheorem{proposition}[theorem]{Proposition}
\newtheorem{lemma}[theorem]{Lemma}
\newtheorem{corollary}[theorem]{Corollary}

\newtheorem{property}{Property}
\newtheorem{remark}[theorem]{Remark}

\numberwithin{equation}{section}

\newcommand{\Z}{\mathbf{Z}}
\newcommand{\R}{\mathbf{R}}

\newcommand{\eps}{\varepsilon}
\newcommand{\oo}{\infty}

\newcommand{\ov}{\overline}
\newcommand{\what}{\widehat}
\newcommand{\dw}{\downarrow}

\newcommand{\up}{\uparrow}
\newcommand{\longto}{\longrightarrow}
\newcommand{\pt}{\partial}
\newcommand{\Id}{I\!d}

\newcommand{\om}{\omega}
\newcommand{\hel} 
{
\raisebox{-.25pt}{\!\!
\hskip2.5pt{\vrule height6.5pt width.5pt depth0pt}
\hskip-3.5pt\vbox{\hrule height.5pt width6.5pt depth0pt}
\hskip-.5pt
}
\!}
\newcommand{\restr}{\hel}
\newcommand{\pf} {\,\!_\#\,} 
\newcommand{\pb}{\,\!^\#\,}

\newcommand\rest[1]{
\hskip2pt
{\vrule height8pt width.5pt depth5pt}
\hskip1.5pt
{\scriptsize \raisebox{-4.5pt}{\ensuremath{#1}}}
}

\newcommand{\bks}{\,\backslash\,}
\newcommand{\be}{\begin{equation}}
\newcommand{\ee}{\end{equation}}

\DeclareMathOperator{\SO}{SO}
\DeclareMathOperator{\sign}{sign}
\DeclareMathOperator{\supp}{supp}

\DeclareMathOperator{\Int}{Int}

\newcommand\D{\mathcal{D}}
\newcommand\lt{\left}
\newcommand\rt{\right}

\def\void{\varnothing}

\def\llb{\llbracket}
\def\rrb{\rrbracket}

\DeclareMathOperator{\diam}{diam}
\DeclareMathOperator{\inter}{Int}
\DeclareMathOperator{\vect}{span}

\def\s{\sigma}

\def\a{\alpha}

\def\x{\xi}
\def\bs{\boldsymbol}
\def\mc{\mathcal}
\def\Q{\mathcal{Q}}
\def\FF{\mathscr{F}}
\def\PP{\mathscr{P}}
\def\RR{\mathscr{R}}

\def\MM{\mathbb{M}}

\def\WW{\mathbb{W}}
\def\I{\mathbb{I}}
\newcommand{\Vhi}{\bs{\Phi}}

\def\H{\mathcal{H}}

\def\de{\partial}

\def\pl{{/\!\!/}}

\title{Strong approximation in $h$-mass of rectifiable currents under homological constraint}
\author{A. Chambolle\footnote{ CNRS, CMAP, \'Ecole Polytechnique CNRS UMR 7641, Route de Saclay, F-91128 Palaiseau Cedex France, email: antonin.chambolle@cmap.polytechnique.fr}   \and L. Ferrari\footnote{CMAP, \'Ecole Polytechnique, CNRS UMR 7641, Route de Saclay, F-91128 Palaiseau Cedex France, email: luca.ferrari@polytechnique.fr}\and B. Merlet\footnote{Laboratoire P. Painlev\'e, CNRS UMR 8524, Universit\'e de Lille, F-59655 Villeneuve d'Ascq Cedex, France and Team RAPSODI, Inria Lille - Nord Europe, 40 av. Halley, F-59650 Villeneuve d'Ascq, France, email: benoit.merlet@univ-lille.fr}}
\date{}
\begin{document}
\maketitle
\begin{abstract}
Let $h:\R\to\R_+$ be a lower semi-continuous subbadditive and even function such that $h(0)=0$ and $h(\theta)\geq \alpha|\theta|$ for some $\alpha>0$. The $h$-mass of a $k$-polyhedral chain $P=\sum_j\theta_j\llb\sigma_j\rrb$ in $\R^n$ ($0\leq k\leq n$) is defined as $\MM_h(P):=\sum_j h(\theta_j)\,\H^k(\sigma_j)$. If $T=\tau(M,\theta,\xi)$ is a $k$-rectifiable chain, the definition extends to $\MM_h(T):=\int_M h(\theta) \, d\H^k$. Given such a rectifiable flat chain $T$ with $\MM_h(T)<\oo$ and $\pt T$ polyhedral,  we prove that for every $\eta>0$, it decomposes as $T=P+\pt V$ with $P$ polyhedral, $V$ rectifiable, $\MM_h(V)<\eta$ and $\MM_h(P)< \MM_h(T)+\eta$. In short, we have a polyhedral chain $P$  which strongly approximates $T$ in $h$-mass and preserves the homological constraint  $\pt P=\pt T$.\\
These results are motivated by the study of approximations of $\MM_h$ by smoother functionals but they also provide explicit formulas for the lower semicontinuous envelope of $T\mapsto \MM_h(T)+\I_{\pt S}(\pt T)$ with respect to the topology of the flat norm. 
\end{abstract}
\section{Introduction}\label{Section: Introduction}
Let $n\geq 0$ be an integer. For $k\in\{1,\cdots,n\}$ we note $\RR_k(\R^n)$ the space of rectifiable currents $T$ with dimension $k$ in the ambient space $\R^n$ and with finite mass $\MM(R)<\oo$. Every $T=\tau(M,\xi,\theta)\in \RR_k(\R^n)$ writes as 
\[
\lt<T,\om\rt>=\int_M \theta(x) \lt<\xi(x),\om(x)\rt>\,d\H^k(x)\qquad \mbox{for any smooth, compactly supported $k$-form $\om$}.
\]  
Here, $M\subset\R^n$ is a countably $\H^k$-rectifiable set oriented by $\xi:M\to\Lambda_k(\R^n)$ where $\xi(x)$ is $\H^k$-almost everywhere a simple unit $k$-vector and $\theta :M\to \R$ is a Borel measurable multiplicity function.\\
We fix a measurable even function $h:\R\to\R_+$ with $h(0)=0$ and we define the $h$-mass of $T=\tau(M,\xi,\theta)\in \RR_k(\R^n)$ as
\[ \MM_h(T):=\int_{M} h(\theta) d\H^k.\]
Given a $k$-current $S\in\RR_k(\R^n)$, the following optimization problem can be considered.
\be
\label{OptPb}
\inf\lt\{\MM_h(T) :R\in \RR_k(\R^n), \, \pt R=\pt S \rt\}.
\ee 
Such problem appears in the context of branched transportation with $k=1$, see~\cite{Xia1,Xia2,PS2006,BCS2009}. 
 An important family of examples is provided by the choice $h(\theta)=|\theta|^\alpha$, $0\leq\alpha\leq1$. For $\a=1$ (that is $h(\theta)=|\theta|$) we have $\MM_h(T)=\MM(T)$ and we recover the mass minimizing Plateau problem whereas for $\a=0$ (that is $h(\theta)=1$ if $\theta\neq0$ and $h(0)=0$) we obtain the size minimizing Plateau problem. \medskip

Let us first discuss the question of the existence of a minimizer for~\eqref{OptPb}. We assume that the support of $S$ is compact, that is $\supp S\subset \ov{B_\lambda}$ for some $\lambda>0$ so that using the orthogonal projection onto $\ov{B_\lambda}$, we can restrict the set of candidates for problem~\eqref{OptPb} to rectifiable currents with $\supp R\subset \ov{B_\lambda}$. In order to obtain the existence of a minimizer for~\eqref{OptPb} by the direct method of the Calculus of Variations, we are looking for two properties: 
\begin{itemize}
\item[(i)]  the sequential compactness of the set 
\[
\Lambda_C := \{T\in \RR_k(\R^n):\MM_h(T)\leq C,\ \supp T\subset \ov{B_\lambda}\},
\]
for $C\geq 0$;
\item[(ii)]
the lower semicontinuity of the functional $\MM_h$.  
\end{itemize}
A natural topology for these properties to hold is the one introduced by Whitney~\cite{Whitney57}. Namely the flat norm of a $k$-current $T$ is defined as
\[
\WW(T):=\inf \,\{ \MM(U) +\MM(V)\},
\]
where the infimum runs over all possible decompositions $T=U+\pt V$. In the sequel any convergence is considered in the latter norm. Furthermore we denote with
\[
\mbox{$\FF_k(\R^n)$ the space of $k$-flat chains in $\R^n$,}
\]
that is the closure of $\RR_k(\R^n)$ in the flat norm topology. It is not difficult to see that for $\MM_h$ being well defined and lower semicontinuous with respect to flat convergence, we need: 
\be
\label{condf}
\mbox{$h(0)=0$, \quad $h$ even, lower semicontinuous and subadditive.}\smallskip\\
\ee
Here, we also require the $h$-mass to control the usual mass of currents.
\be
\label{condf2}
\mbox{There exists $\alpha>0$ such that $h(\theta)\geq \alpha |\theta|$ for $\theta\in \R$}.
\ee
In the recent paper~\cite{CRMS2017}, it is established that under conditions in~\eqref{condf}, $\MM_h$ is lower semi-continuous on $\RR_k(\R^n)$. The result is more precise. Let us recall that a $k$-polyhedral current is a $k$-rectifiable current which writes as a finite sum 
\[
P=\sum_j\theta_j\llb\sigma_j\rrb.
\] 
The $\theta_j\in \R$ are multiplicities, the $\sigma_j$ are oriented $k$-polyhedrons and for every $j$,  $\llb\sigma_j\rrb$ denotes the integration of smooth $k$-differential forms over $\sigma_j$. We note $\PP_k(\R^n)\subset \RR_k(\R^n)$ the space of $k$-polyhedral currents. In~\cite{CRMS2017}, the authors introduce the lower semicontinuous envelope of $\MM_h$ restricted to $\PP_k(\R^n)$ with respect to the flat convergence:
\[\Vhi_h(T):=\inf \lt\{\liminf_{j\up\oo}\ \MM_h(\PP_j) : (P_j)\subset \PP_k(\R^n),\, P_j \to T\rt\}.\]
They prove that under assumption~\eqref{condf}, their holds  $\Vhi_h=\MM_h$ on $\RR_k(\R^n)$. This result is also stated in~\cite[Sec 6.]{White1} in the context of $G$-valued flat chains with a sketch of proof. Assuming moreover~\eqref{condf2} and
\be\label{condf3}
 \mbox{ $h$ is non-decreasing on $(0,+\oo)$ with $\lim_{\theta\dw 0} h(\theta)/\theta=+\oo$,}
 \ee 
it is established that $\Vhi_h\equiv +\oo$ on $\FF_k(\R^n)\setminus \RR_k(\R^n)$ (see~\cite[Prop. 2.7]{CRMS2017}). This proves the compactness of the sets $\Lambda_C$. 

\begin{remark}
Under~\eqref{condf} the condition~\eqref{condf2} is equivalent to~$\lim_{\theta\up \oo} h(\theta)/\theta>0$. If this condition fails then the compactness of a minimizing sequence for problem~\eqref{OptPb} is not clear. In fact, in general minimizers do not exist in the set of rectifiable currents (see~\cite[example of Sec. 1]{dPH2003}).  Nevertheless, in the special case $k=1$ and $\MM(\pt S)<\oo$, we can assume that~\eqref{condf2} holds true.\footnote{ Indeed, using Smirnov decomposition~\cite{Smirnov1994}, any candidate $R=\tau(M,\xi,\theta)$ for problem~\eqref{OptPb} decomposes as $R=R'+R_0$ with $\pt R_0=0$, $\MM(R)=\MM(R')+\MM(R_0)$ and $\MM(R')$ minimal.  We then have $R'=\tau(M,\xi,\theta')$ with  $|\theta'|\leq \MM(\pt S)/2=:q$ and 
$\MM_h(R')\leq\MM_h(R)$. As a consequence, we can restrict the set of candidates for problem~\eqref{OptPb} to rectifiable currents $R=\tau(M,\xi,\theta)$ such that $|\theta|\leq q$. Modifying $h$ in $\R\setminus [-q,q]$ we can assume\eqref{condf2}.}\medskip
\end{remark}

Notice that  the homological constraint $\pt T=\pt S$ does not appear in the definition of $\Vhi_h$. In this note, we consider the lower semicontinuous envelope of $\MM_h$ restricted to the set of polyhedral currents satisfying $\pt P=\pt S$. Let us assume $\pt S$ to be a polyhedral current and let us note $\I^S$ the (convex analysis') indicatrix function of the set $\{T\in \FF_k(\R^n ): \pt T=\pt S\}$, that is 
\[
\I^S(T)=\begin{cases}\quad 0&\mbox{if $T\in \FF_k(\R^n)$ with $\pt T=\pt S$,}\\
+\oo&\mbox{in the other cases.}
\end{cases}
\]
We note
\[\Vhi_h^S(T):=\inf \lt\{\liminf_{j\up\oo}\ \MM_h(P_j) +\I^S(P_j): (P_j)\subset \PP_k(\R^n),\, P_j \to T\rt\}.\]
We obviously have $\Vhi_h^S\geq \Vhi_h$ and by continuity of the boundary operator under flat convergence, we also have $\Vhi_h\geq \I^S$. Hence, 
\be
\label{geqineq}
\Vhi_h^S \geq \Vhi_h +\I^S.
\ee
The opposite inequality follows from the following strong polyhedral approximation result whose proof is our main purpose.
\begin{theorem}\label{thm1}
Let us assume that $h:\R\to\R_+$ satisfies~\eqref{condf}\eqref{condf2}. Let $R\in \RR_k(\R^n)$ with $\MM_h(R)<\oo$  and $\pt R\in \PP_{k-1}(\R^n)$, then  
for every $\eta>0$, we have the decomposition $R=P+\pt V$ for some  $P\in \PP_k(\R^n),\ V\in \RR_{k+1}(\R^n)$ satisfying the estimates
\[
\MM_h(P)\ <\ \MM_h(R)+\eta\quad\mbox{and}\quad \MM_h(V)\ <\ \eta.
\]
Moreover, $\supp V\subset \supp R+B_\eta$.
\end{theorem}
If we drop assumption~\eqref{condf2} but assume that $\MM(R)<\oo$, the result still holds true. Indeed, applying Theorem~\ref{thm1} to $\tilde h(\theta):=|\theta|+ h(\theta)$ (that is $\MM_{\tilde h}=\MM+\MM_h$) and using the lower semi-continuity of $\MM$ under flat convergence, we obtain:
\begin{corollary}\label{coro_thm1}
Assume that $h:\R\to\R_+$ satisfies~\eqref{condf} and let $R\in \RR_k(\R^n)$ with $\MM_h(R)+\MM(R)<\oo$  and $\pt R\in \PP_{k-1}(\R^n)$, then  
for every $\eta>0$, we have the decomposition $R=P+\pt V$ for some  $P\in \PP_k(\R^n),\ V\in \RR_{k+1}(\R^n)$ satisfying the estimates
\[
\MM_h(P)\ <\ \MM_h(R)+\eta,\qquad \MM(P)\ <\ \MM(R)+\eta\quad\mbox{and}\quad \MM(V)+\MM_h(V)\ <\ \eta.
\]
Moreover, $\supp P \cup \supp V\subset \supp R+B_\eta$.
\end{corollary}
\noindent
Taking into account~\eqref{geqineq} and the results of~\cite{CRMS2017}, we obtain an explicit form for $\Vhi_h^S$. 
\begin{corollary}
\label{coroPhiS}
Under condition~\eqref{condf}, we have $\Vhi_h^S=\Vhi_h+\I_h^S$. In particular, $\Vhi_h^S=\MM_h+\I^S$ on $\RR_k(\R^n)$. Moreover, under conditions~\eqref{condf}, \eqref{condf2} and~\eqref{condf3}, for $T\in \FF_k(\R^n)$
\[
\Vhi_h^S(T) =
\begin{cases} \MM_h(T) &\mbox{if $T\in \RR_k(\R^n)$ with $\pt T=\pt S$,}\\
\ +\oo&\mbox{in the other cases.}
\end{cases}
\]
\end{corollary}

\subsection{Motivation}

When it comes to numerical simulations, it is often convenient to substitute for~\eqref{OptPb} a family of approximate variational problems with better differentiation properties: for $\eps\in(0,1]$,
\be
\label{OptPbeps}
\inf\lt\{\MM_h^\eps(T_\eps) :T_\eps\in \D_k(\R^n), \, \pt T_\eps=\pt S_\eps\rt\}. 
\ee
Here the boundary condition is provided by a family of currents $\{S_\eps\}$ which are given mollifications of $S$ and such that $S_\eps\to S$ as $\eps\dw 0$.
This strategy is  implemented in, e.g.~\cite{Santambroggio,OS2011,Monteil2017,Bon_Lem_San,CFM2016,CFM2017}. The asymptotic equivalence between the approximate variational problem and~\eqref{OptPb} follows from the (expected) $\Gamma$-convergence of the family $\{\MM_h^\eps\}$ towards $\MM_h$ as $\eps\dw 0$. In particular the upper bound part of the $\Gamma$-convergence asserts that for any $T=\tau(M,\xi,\theta)\in \RR_k(\R^n)$ with $\pt T=\pt S$, there exists a family $\{T_\eps\}$ with $\pt T_\eps=\pt S_\eps$ such that  $T_\eps\to T$ and 
\[
\MM_h(T)\geq \limsup_{\eps\dw0} \MM_h^\eps(T_\eps).
\] 
Usually, the construction of such a recovery family $\{T_\eps\}$ is easier when $M$ is a smooth manifold and $\theta$ is smooth. In fact, the family of functionals $\{\MM_h^\eps\}$ is designed for this. A method for building $\{T_\eps\}$ in the general case consists in reducing to this special case: we first approximate $T$ with a smooth or piecewise smooth rectifiable current: here, a polyhedral current. More precisely, the polyhedral current $P$ should be close to $T$  in flat distance with $\MM_h(P)\leq \MM_h(T)+o(1)$. These conditions are not sufficient. Indeed, having in mind the constraint $\pt T_\eps=\pt S_\eps$, we also need a constraint on $\pt P$. If $S$ is a polyhedral current, we can impose $\pt P=\pt S$. In this case, the approximation theorem~\ref{thm1} fits our needs. The above result extends to the case of $\pt S$ being a piecewise $C^1$-cyclic $(k-1)$-manifold if we allow $P$ to be a piecewise $C^1$-current, but this is far from enough. For usual branched transportation problems, the constraint $\pt S$ may be supported on a set with dimension larger than $(k-1)$. A natural requirement is then to assume that $\pt S$ can be deformed into a polyhedral current with small energy expense. We assume:
\be
\label{condbdry}
\exists \{\Sigma_\eps\}\subset \PP_k(\R_n),\ \exists \{Z_\eps\}\subset \RR_k(\R_n)\mbox{ with }  \pt S + \pt Z_\eps =\pt \Sigma_\eps \mbox{ and }\MM_h(Z_\eps) =o (1).
\ee
It is then convenient to define the approximate constraint $S_\eps$ in~\eqref{OptPbeps} as a mollification of $\Sigma_\eps$. Applying Theorem~\ref{thm1} to $T+Z_\eps$, we get the following.
\begin{corollary}
\label{coro1}
Let $h$ satisfying~\eqref{condf}\eqref{condf2}, let $S\in \RR_k(\R^n)$ and assume~\eqref{condbdry}. Then for any $k$-current with $\pt T=\pt S$ and $\MM_h(T)<\oo$, there exist 
$\{P_\eps\}\subset \RR_k(\R^n)$, $\{V_\eps\}\subset \RR_{k+1}(\R^n)$ with 
\[
\pt P_\eps=\pt \Sigma_\eps,\qquad T=P_\eps+\pt V_\eps,\qquad \MM_h(V_\eps)=o(1) \qquad \MM_h(P_\eps)\leq \MM_h(T)+o(1).
\] 
\end{corollary}

\subsection{A possible method of proof}
\label{Smethodofproof}
Let us first describe a proof of a weaker version of Theorem~\ref{thm1}, where we assume $\MM_h\leq \beta \MM$ for some $\beta>0$. \medskip

\noindent
{\it Step 1.} The first step is given in~\cite{CRMS2017}.
\begin{proposition}\cite[Proposition 2.6]{CRMS2017}\label{Prop 2.6}
Assume that $h$ satisfies~\eqref{condf}\eqref{condf2}, let $R\in \RR_k(\R^n)$ compactly supported with $\MM_h(R)<\oo$ and let $\eta>0$. There exist $P_1\in \PP_k(\R^n)$, $U_1\in \FF_k(\R^n)$ and $V_1\in F_{k+1}(\R^n)$ such that 
\be
\label{descript1}
R=P_1+U_1+\pt V_1,\quad\mbox{with}\quad  \MM_h(P_1)\ <\ \MM_h(R)+\eta\quad\mbox{and}\quad \MM(U_1)+\MM(V_1)\ <\ \eta.
\ee
\end{proposition}

\noindent
\emph{Step 2. Approximation of $T$ preserving the boundary.} Next, assuming further $\pt T \in \PP_{k-1}(\R^n)$, we decompose $U_1$ as
\be
\label{decompU1}
U_1=P_2+\pt V_2\quad\mbox{with}\quad P_2\in \PP_k(\R^n),\quad  \MM(P_2)+\MM(V_2)\ \leq\ C\, \MM(U_1).
\ee
This decomposition is the consequence of the deformation theorem of Federer and Fleming~\cite{FedFlem60} (see {\it e.g.}~\cite[4.2.9]{federer}, \cite{KP2008}). Indeed, by assumption $\pt U_1=\pt T-\pt P_1\in \PP_{k-1}(\R^n)$ and in this case, the deformation theorem simplifies to~\eqref{decompU1}. Eventually, writing $P=P_1+P_2\in \PP_k(\R^n)$ and $V=V_1+V_2$, we get, the desired decomposition
\[
T=P+\pt V\quad\mbox{with}\quad \MM(V)\ \leq \ \MM(V_1)+\MM(V_2)< (1+C)\eta.
\]
and, using $\MM_h(P_2)\leq \beta \MM(P_2)\leq C\beta\MM(U_1)\leq  C\beta\eta$,
\[ 
\MM_h(P)\ \leq \ \MM_h(P_1)+\MM_h(P_2)\, < \MM_h(T)+(1+C\beta)\eta.\]
This proves Theorem~\ref{thm1} under the assumption $\MM_h\lesssim \MM$.\\

To recover the full Theorem with the same line of proof, we  first need to improve~\eqref{descript1} to have moreover 
\be
\label{improvement}
\mbox{$U_1$ and $V_1$ are rectifiable and\quad} \MM_{h}(U_1)+\MM_{h}(V_1)\ <\ \eta.
\ee
Next, for the second step, we need a $h$-mass version of the classical deformation theorem, namely:
\begin{theorem}\label{thm3}
Let $h:\R\to\R_+$ satisfying~\eqref{condf}\eqref{condf2}, let $R\in \RR_k(\R^n)$ with $\pt R\in \RR_{k-1}(\R^n)$ and $\MM_h(R)+\MM_h(\pt R)<\oo$ and let $\eps>0$. There exist $P\in \PP_k(\R^n)$, $U\in \RR_k(\R^n)$ and $V\in \RR_{k+1}(\R^n)$ such that 
\[
R=P+ U +\pt V,\qquad \supp P\,\cup\, \supp U\, \cup\, \supp V\, \subset\, \supp R+ \ov{B}_{\sqrt{n}\eps}.
\]
Moreover, there exists $c=c(n)>0$ such that 
\[
\MM_h(P)\ \leq\ c \,\MM_h(R),\qquad \MM_h(U)\ \leq \ c\, \MM_h(\pt R)\,\eps,\qquad  \MM_h(V)\ \leq \ c\,\MM_h(R)\eps.
\]
Eventually, if $\pt R$ is polyhedral, so is $U$.
\end{theorem}
Applying the theorem with $R=U_1$ as above and using the subadditivity of $\MM_h$, we obtain the desired result.\\
Unfortunately,~\eqref{improvement}  is not stated in~\cite{CRMS2017}. However, in the proof of~\cite[Proposition 2.6]{CRMS2017} the currents $U_1$ and $V_1$  obtained in~\eqref{descript1}  are rectifiable by construction and with obvious modifications\,\footnote{The idea is to consider Lebesgue points of the function $h(\theta)$ rather than Lebesgue points of $|\theta|$, the function $\theta$ being the multiplicity of $R=\tau(M,\xi,\theta)$.}
we can assume that $U_1$ and $V_1$ satisfy the estimate~\eqref{improvement}. We further remark that the assumption in~\cite{CRMS2017} about $R$ being compactly supported can be removed. Besides, the construction being a sequence of local deformations we can assume 
\[
\supp U_1 \cup \supp V_1 \ \subset \ \supp R+B_\eta.
\]
In conclusion, this scheme provides a proof of~Theorem~\ref{thm1}. \\
Here we propose a different approach based on a local deformation lemma and which we believe to be of independent interest.

\subsection{The case $\MM_h\lesssim\MM$}
Let us now turn our attention to the cases where~\eqref{condf3} fails. First, notice that if 
\be
\label{beta}
\beta:= \sup_{\theta>0} \dfrac{h(\theta)}\theta=\limsup_{\theta\dw 0}  \dfrac{h(\theta)}\theta<\oo,
\ee
then the set $\Lambda_C$ is not closed. In fact,
\[
\ov{\{R\in \RR_k(\R^n) :\MM_h(R)<\oo\}} = \{T\in \FF_k(\R^n):\MM(T)<\oo\} =: \ \FF_k^\MM(\R^n).
\]
The domain of $\Vhi_h$ is then the whole space of $k$-flat chains with finite mass. Assuming moreover, that the $\limsup$ in~\eqref{beta} is a limit, that is 
\be\label{beta00}
\beta=\lim_{\theta\dw0}h(\theta)/\theta,
\ee
we expect that the lower semicontinuous envelope of $\MM_h$ has the explicit form:
\be\label{wMMh}
\Vhi_h(T)=\widehat{\MM}_h(T):=\MM_h(R)+\beta \MM(T'),
\ee
where $T\in \FF_k(\R^n)$, is decomposed into its rectifiable and ``diffuse'' parts,  $T=R+T'$ 
(this decompositon  is built in Section~\ref{Sbeta00}). Notice that from~\eqref{condf2}, \eqref{beta00} and the subadditivity of $h$, we have 
\[
\alpha\MM  \leq \widehat \MM_h  \leq \beta \MM.
\]
In the setting we have the following strong approximation result. 
\begin{theorem}\label{thm2}
Let us assume that $h:\R\to\R_+$ satisfies~\eqref{condf}, \eqref{condf2} and~\eqref{beta}. Let $T=R+T'$, with $R\in \RR_k(\R^n)$, $T'\in \FF_k(\R^n)$ and $\MM_h(R)+\MM(T')<\oo$.
For every $\eta>0$ there exist $P\in \PP_k(\R^n)$, $U\in \FF_k(\R^n)$ and  $V\in \FF_{k+1}(\R^n)$ such that $T=P+U+\pt V$ and with the estimates
\[
\MM_h(P)\ <\ \MM_h(R)+\beta \MM(T')+ \eta\quad\mbox{and}\quad \MM(U) + \MM(V)\ <\ \eta.
\]
Moreover, 
if $\pt T\in \PP_{k-1}(\R^n)$, we can take $U=0$.
\end{theorem}
\noindent
The proof of Theorem~\ref{thm2} that we propose is very close to the two steps proof already described in Subsection~\ref{Smethodofproof}. However, since $\widehat{\MM}_h\simeq \MM$, there is no point here to improve the classical deformation theorem. The situation is more simple than in Theorem~\ref{thm1}.\\
In order to establish~\eqref{wMMh}  we should prove that $\widehat{\MM}_h$ is lower semicontinuous with respect to the flat norm topology. This is out of the scope of the present note but we believe that this can be done with a method based on slicing as in~\cite{dPH2003,CRMS2017}.


\subsubsection*{Organization of the note}
In the next section, we set some notation and we recall basic facts about rectifiable currents, push-forward  by Lipschitz maps and homotopy formulas.  In Section~\ref{S3} we prove a local deformation theorem: Lemma~\ref{lem1}. 
Theorem~\ref{thm1} is established in Section~\ref{S4}. Eventually we prove Theorem~\ref{thm2} in the last (short) section.
\section{Preliminaries and notation}
\label{S2}
\subsection{Currents}
For the notions about differential forms, currents and rectifiable currents we refer to~\cite{federer,KP2008}. We note $\D^{j}(O)$ the space of smooth and compactly supported $j$-differentiable forms and $\D_j(\R^n)$ the space of $j$-currents in $\R^n$. To avoid discussion of particular cases, we adopt the conventions: $\D^{-1}(\R^n)=\D^{n+1}(\R^n)=\{0\}$ and $\D_{-1}(\R^n)=\D_{n+1}(\R^n)=\{0\}$ (and the same for all the possible subspaces).\\
The boundary operator $\pt :\D_j(\R^n)\to \D_{j-1}(\R^n)$ is defined by the duality formula $\lt<\pt T,\om \rt>:=\lt<T,d\om\rt>$ for $\om\in \D_{j-1}(\R^n)$.\\
The comass of a $j$-covector $\zeta\in\Lambda^j(\R^n)$ is defined as $|\zeta|_*:=\max\lt<e,\zeta\rt>$ where $e$ ranges over the set of unit simple $j$-vectors. The mass of a current $T\in\D_j(\R^n)$ is defined as $\sup \lt<T,\om\rt>$ where the supremum is taken over every $\om\in \D^j(\R^n)$ with $\sup|\om(x)|_*\leq 1$. Whitney's flat norm~\cite{Whitney57} of a current $T\in \D_j(\R^n)$ is defined as 
\[
\WW(T):=\inf\, \{\MM(T-\pt V)+\MM(V):V\in \D_{j+1}(\R^n)\}. 
\]

\subsubsection*{Rectifiable currents}
Here we deal with finite mass currents, which can be seen as Radon measures with values into $\Lambda_j(\R^n)$. More specifically we deal with the space $\R_j(\R^n)$ of $j$-rectifiable currents with finite mass. Every $T\in \R_j(\R^n)$ is of the form $T=\tau(M,\theta,\xi)$ where:
\begin{itemize}
\item  $M$ is a countably $j$-rectifiable set;
\item $\theta\in L^1(\H^j\restr M)$ is the multiplicity function;
\item $\xi\in L^\oo(\H^j\restr M,\Lambda_j(\R^n))$ takes values in the set of unit simple $j$-vectors and for $\H^j$-almost every $x$, $\xi(x)$ generates the approximate tangent space of $M$ at $x$.
\end{itemize}
With this notation, $T=\tau(M,\theta,\xi)$ is defined as
\[
\lt<T,\om\rt>=\int_M \theta(x) \lt<\xi(x),\om(x)\rt> \, d\H^j(x),\quad\mbox{for every }\om\in\mc{D}^j(\R^n)=C^\oo_c(\R^n,\Lambda^j(\R^n)).
\]
From the point of view of measures,  we have the polar decomposition  $T=\|T\| \sign(\theta) \xi$ with $\|T\|=\H^k\restr |\theta|$ and $\MM(T)=\|T\|(\R^n)=\int_M|\theta|\,d\H^j$.\\
It is usual to consider the restriction of $T\in \D_j(\R^n)$ to an open subset of $\R^n$, but when $T$ has finite mass, we can consider the restriction of $T$ to any Borel set $B\subset \R^n$. In particular, if $T=\tau(M,\theta,\xi)$, we have 
\[
\mbox{$T\restr B=\tau(M\cap B,\theta,\xi)$ \quad and \quad $T=T\restr B+T\restr B^c$.}
\]

\subsubsection*{$h$-mass of rectifiable currents}
For every even function $h:\R\to\R_+$ satisfying $h(0)=0$ we can consider the energy of $T=\tau(M,\theta,\xi)\in \RR_j(\R^n)$ defined as,
\[
\MM_h(T):=\int_M h(\theta(x)) \, d\H^j(x).
\]
In the sequel $h$ is always subadditive. In this case we have $\MM_h(T_1+T_2)\leq \MM_h(T_1)+\MM_h(T_2)$ and assuming moreover that $h$ is lower smicontinuous, this extends to countable sum: if $T=\sum_l T_l$ then $\MM_h(T)\leq\sum_l \MM_h(T_l)$.

\subsubsection*{Polyhedral currents and the constancy theorem}

When $\s$ is an oriented $j$ polyhedron, we note $\llb\s\rrb$ the current corresponding to the integration of differential forms on $\s$. These currents generate the space of polyhedral current $\PP_j(\R^n)\subset \R_j(\R^n)$. In the sequel, in order to show that some currents are polyhedral chains we will use the following constancy theorem which is a simple consequence of~\cite[Sec. 4.2.3]{federer}.
\begin{lemma}\label{lemConstancy}
Let $T\in {\mc D}_j(\R^n)$, assume that $\supp T\subset X$ and $\supp \pt T\subset Y$ where  $X$ is a finite union of closed $j$-polyhedrons and $Y$ in a finite union of $(j-1)$-polyhedrons, then $T \in \PP_j(\R^n)$. 
\end{lemma}

\subsubsection*{Push forward of rectifiable currents and homotopy formula}
Let $u\in C^\oo (\R^m,\R^n)$ be a proper mapping and $T\in \D_k(\R^m)$, then the pushforward of $T$ by $u$ is the current 
$u\pf T\in \D_k(\R^n)$ defined as  
\[
\lt<u\pf T,\om\rt>\, =\, \lt<T, u\pb \om\rt>,\qquad \mbox{for $\om\in\D^k(\R^n)$},
\]
where $u\pb \alpha(x) d x_{i_1}\wedge\cdots\wedge d x_{i_k}:=\alpha(u(x)) d u_{i_1}(x)\wedge\cdots\wedge d u_{i_k}(x)$. By duality, we have $\pt [u\pf T]=u\pf[\pt T]$. If $T=\tau(M,\theta,\xi)\in \RR_k(\R^m)$ then the formula extends to $u$ Lipschitz continuous and proper and we have the close form $u\pf T=\tau(\mc{M},\Theta,\Xi)$,  where $\mc{M}=u(M)\subset \R^n$ is a countably $k$-rectifiable set,
 $\Xi\in L^\oo(\H^h\restr u(M),\Lambda^k(\R^n))$ takes values in the set of simple unit vectors and for $\H^k$-almost any $y\in\mc{M})$, $\Xi(y)$ generates the (approximate) tangent space to $\mc{M}$ at $y$. Finally, the multiplicity is given by   
\[
\Theta(y):=\sum_{x\in u^{-1}(y)} \theta(x)\eps(x).
\]
with $\eps(x)\in\{\pm 1\}$ given by,  
\[
Du(x)v_1\wedge\cdots\wedge Du(x)v_k =\eps(x) |Du(x)v_1\wedge\cdots\wedge Du(x)v_k|\,\Xi(x), \qquad\mbox{where }\xi(x)=v_1\wedge\cdots\wedge v_k. 
\]
Using the above formula to express $\MM_h(u\pf T)$ and using the change of variable $y=u(x)$, the subadditivity and lower semicontinuity of $h$ lead to,
\be\label{estimPFLip}
\MM_h(u\pf T)\, \leq\, \int_{M} |Du|(x)^k h(\theta(x))\, d\H^k(x)\, \leq\, \|Du\|_{L^\oo(M)}^k \MM_h(T). 
\ee
Let $z:[0,1]\times \R^n\to \R^n $ be a proper Lipschitz mapping, then for $T\in \RR_k(\R^n)$, we have the homotopy formula (see~\cite[4.1.9]{federer},~\cite[2.2.3. Prop. 4]{GMS1} or \cite[7.4.3]{KP2008}),
\be
\label{homform}
z(1,\cdot) \pf T - z(0,\cdot)\pf  T= \pt\lt[ z\pf(\llb(0,1)\rrb\times T) \rt]+z\pf(\llb(0,1)\rrb\times \pt T). 
\ee
This formula is the basis of the deformation method (with $z(0,\cdot)=\Id$, $z(1,\cdot)=u$).

\subsubsection*{Convention} In the sequel $C$ denotes a non negative constant that may only depend on the ambient dimension $n$ and that may change from line to line.\\
When $Q$ is a $k$-cube with side length $\ell(Q)$ and $\lambda>0$, we note $\lambda Q$ the dilated $k$-cube with same orientation and same center as $Q$ but with side length $\lambda\ell(Q)$.

\section{A local deformation lemma}
\label{S3}
The building block of our proof is the local deformation lemma, Lemma~\ref{lem1} below. Let us first introduce some notation.
Let $\delta>0$, $x\in\R^n$ and $(e_1,\cdots,e_n)$ be an orthonormal basis of $\R^n$. Let $Q_0=x_0+\{\sum t_i e_i,  0<t_i<\delta\}$ be an open cube with side length $\delta>0$. We introduce the collection of its translates:
\[
\Q^{(n)}:=\lt\{\lt(\sum  a_i e_i\rt)+ Q_0  :a=(a_1,\cdots,a_n) \in \delta \Z^n \rt\}.
\] 
For $0\leq j\leq n$, we also note $\Q^{(j)}$ the set of  \emph{relatively open} $j$-faces of the cubes of $\Q^{(n)}$. For instance:
\begin{itemize}
\item $\Q^{(0)}$ is the set of vertices 
$\{x_0+\delta \sum l_i e_i : (l_1,\cdots,l_n)\in \Z^n\}$;
\item $\Q^{(1)}$ is the set of open segments $(y,z)$ with $y,z\in \Q^{(0)}$ and $|y-z|=\delta$;
\item$\Q^{(2)}$ is the set of squares $I\times J$ with $I,J\in \Q^{(1)}$ and $\ov I\cap \ov J=\{y\}$ for some $y\in \Q^{(0)}$;
\item and so on \dots
\end{itemize}
By construction, 
\[
\Q:=\Q^{(0)}\cup\cdots\cup\Q^{(n)}
\]
form a partition of $\R^n$. For $Q\in \Q^{(k)}$, $k\in\{0,\cdots,n\}$, we introduce the closed set
\[
\Sigma_Q:=\ov Q\,\cup\,\bigcup \lt\{\ov L : L\in \Q^{(n)},\,  Q\not\subset\pt L=\void \rt\}
\]
and its open complement 
\[
\om_Q := \R^n\bks \Sigma_Q=\lt\{M\in\Q: Q\subset\pt M \rt\}.
\]
Notice that for $Q\in \Q^{(n)}$, $\om_Q=\void$, $\Sigma_Q=\R^n$. For later use, we notice that
\be
\label{caractomQ}
\om_Q\, =\, \cup \{L\in \Q^{(j)}:k<j\leq n,\,Q\subset\pt L \}.
\ee
 Some examples of sets $\om_Q$ and $\Sigma_Q$ are illustrated in Figures~\ref{fig1},~\ref{fig2},~\ref{fig3} in the ambient spaces $\R$, $\R^2$ and $\R^3$.

\begin{figure}[H]
\centering
\begin{tikzpicture}
\fill[ultra thick, color=gray] (-1,0) node{$|$} -- (0,0) node{$|$}  -- (4,0) node{$|$} -- (5,0) node{$|$};
\draw[very thick,dashed, color=RedOrange] (-1.5,0) --(1,0) (3,0)-- (5.5,0)node[anchor=south west,color=black]{$\R$};
\draw[very thick, color=Cerulean](1,0)--(3,0);
\draw[very thick, color=RedOrange] (-.5,0) --(1,0) node{$]$} (2,0)node{$|$} (3,0)node{$[$} -- (4.5,0);
\filldraw[RedOrange] (2,0) circle (1.5pt) node[anchor=south east] {$Q$};
\draw (3,0)node[anchor=north east] {$\omega_Q$};
\draw (4.5,0)node[anchor=north west] {$\Sigma_Q$};
\end{tikzpicture}
\caption{In the ambient space $\R$ we represent the open set $\omega_Q$ in blue and  its closed complementary $\Sigma_Q$ in orange for some $Q\in \Q^{(0)}$}\label{fig1}
\end{figure}
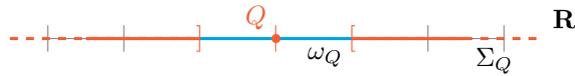

\begin{figure}
\centering
\begin{tikzpicture}
\fill[fill=white] (-.5,-.5) rectangle (4.5,4.5);
\fill[pattern=north east lines, pattern color=white!40!RedOrange] (-.4,1.6) rectangle (4.4,3.4);
\filldraw[fill=white, draw=black] (1,2) rectangle (3,3);
\draw[very thin,dashed,color=gray] (-.5,1.5) grid (4.5,3.5);
\fill[pattern=north west lines, pattern color=Cerulean] (1,2) rectangle (3,3);
\draw[very thin,dashed,color=gray] (-.5,1.5) grid (4.5,3.5);
\filldraw[RedOrange, very thick] (2,2) -- (2,2.5) node[anchor= east] {$Q$} -- (2,3);
\draw (3,2)node[anchor=south east] {$\omega_Q$};
\draw (4,1.6)node[anchor=south east] {$\Sigma_Q$};
\draw (5,4)node[anchor=south west] {$\R^2$};
\end{tikzpicture}
\qquad
\begin{tikzpicture}
\fill[pattern=north east lines, pattern color=white!40!RedOrange] (-.4,-.4) rectangle (4.4,4.4);
\filldraw[fill=white, draw=black] (1,1) rectangle (3,3);
\draw[very thin,dashed,color=gray] (-.5,-.5) grid (4.5,4.5);
\fill[pattern=north west lines, pattern color=Cerulean] (1,1) rectangle (3,3);
\filldraw[RedOrange] (2,2) circle (1.5pt) node[anchor=south east] {$Q$};
\draw (3,1)node[anchor=south east] {$\omega_Q$};
\draw (4,0)node[anchor=south east] {$\Sigma_Q$};
\end{tikzpicture}
\caption{In $\R^2$ we draw the set $\omega_Q$, in striped blue, and $\Sigma_Q$, in striped orange, on the left in the case $Q\in \Q^{(1)}$ (also in orange as it belongs to $\Sigma_Q$); on the right, with same color codes, the sets associated with $Q\in \Q^{(0)}$.}\label{fig2}
\end{figure}
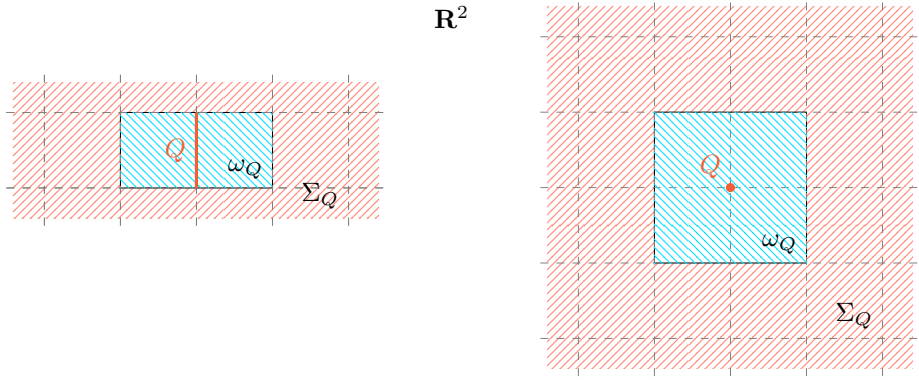

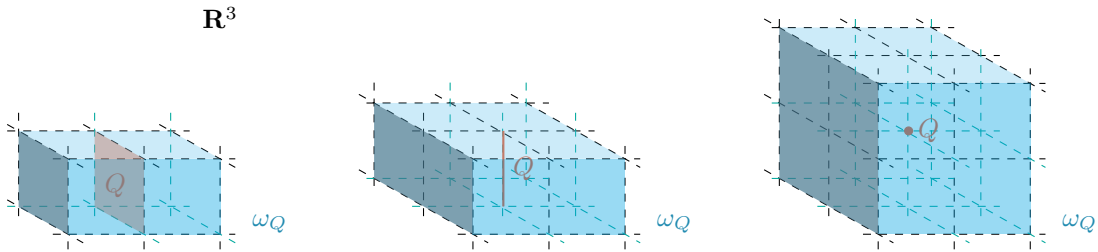
\begin{figure}[H]
\centering
\begin{tikzpicture}
\draw[ dashed,color=black,xshift=2cm, yshift=-1.14cm,xstep=1cm] (-.2,-.2) grid (2.2,1.2);
\draw [dashed,color=black,domain=1.2:2.2] plot ({\x 	  	},       	{-.57*\x + 1});
\draw [dashed,color=black,domain=1.2:2.2] plot ({\x 	  	},       	{-.57*\x 	});
\draw [dashed,color=black,domain=.5:1.9] plot ({1.35 	 },       	{\x-1.4});
\draw [dashed,color=black,domain=.5:2.9 ] plot ({\x	+.7 	},       	{.23});
\draw [dashed,color=black,domain=1.2:2.2] plot ({\x 	+1  	},       	{-.57*\x + 1});
\draw [dashed,color=black,domain=1.2:2.2] plot ({\x 	+2  	},       	{-.57*\x + 1});

\draw [dashed,color=Emerald,domain=.5:2.9 ] plot ({\x	+.7 	},       	{-.77});
\draw [dashed,color=Emerald,domain=1.2:2.2] plot ({\x 	 +1	},       	{-.57*\x });
\draw [dashed,color=Emerald,domain=1.2:2.2] plot ({\x 	+2  	},       	{-.57*\x });
\draw [dashed,color=Emerald,domain=.5:1.9] plot ({3.35 	 	},       	{\x-1.4});
\draw [dashed,color=Emerald,domain=.5:1.9] plot ({2.35 	 	},       	{\x-1.4});
\fill[opacity=.5,color=white!10!RedOrange] (2.35,.23)  -- (3,-.14)  -- (3,-1.14) -- (2.35,-.78)-- cycle; 
\draw[RedOrange] (2.625,-.5)node{$Q$};
\fill[opacity=.4,color=Cerulean] (2,-1.14)  -- (4,-1.14) -- (4,-.14) -- (2,-.14) -- cycle; 
\fill[opacity=.7,color=black!50!Cerulean] (1.35,.23)  -- (2,-.14)  -- (2,-1.14) -- (1.35,-.78)-- cycle; 
\fill[opacity=.3,color=white!40!Cerulean] (1.35,.23)  -- (3.35,.23)  -- (4,-.14)  -- (2,-.14)-- cycle; 
\draw (4.3,-1)  node[anchor= west] {\color{Cerulean!70!black}{$\omega_{Q}$}}; 
\filldraw[black] (4,2)  node[anchor=north] {$\R^3$};

\end{tikzpicture}
\qquad
\begin{tikzpicture}
\draw[ dashed,color=black,xshift=2cm, yshift=-1.14cm,xstep=1cm] (-.2,-.2) grid (2.2,1.2);
\draw [dashed,color=black,domain=.5:2.2] plot ({\x 	  	},       	{-.57*\x + 1});
\draw [dashed,color=black,domain=.5:2.2] plot ({\x 	  	},       	{-.57*\x 	});
\draw [dashed,color=black,domain=.5:1.9] plot ({0.7 	 	},       	{\x-1});
\draw [dashed,color=black,domain=.5:1.9] plot ({1.35 	 },       	{\x-1.4});
\draw [dashed,color=black,domain=.5:2.9 ] plot ({\x	+.7 	},       	{.23});
\draw [dashed,color=black,domain=.5:2.2] plot ({\x 	+1  	},       	{-.57*\x + 1});
\draw [dashed,color=black,domain=.5:2.2] plot ({\x 	+2  	},       	{-.57*\x + 1});
\draw [dashed,color=black,domain=.5:2.9 ] plot ({\x	+.1 	},       	{.60});

\draw [dashed,color=Emerald,domain=.5:2.9 ] plot ({\x	+.7 	},       	{-.77});
\draw [dashed,color=Emerald,domain=.5:2.2] plot ({\x 	 +1	},       	{-.57*\x });
\draw [dashed,color=Emerald,domain=.5:2.2] plot ({\x 	+2  	},       	{-.57*\x });
\draw [dashed,color=Emerald,domain=.5:1.9] plot ({1.7 	 	},       	{\x-1});
\draw [dashed,color=Emerald,domain=.5:1.9] plot ({2.7 	 	},       	{\x-1});
\draw [dashed,color=Emerald,domain=.5:2.9 ] plot ({\x	+.1 	},       	{-.40});
\draw [dashed,color=Emerald,domain=.5:1.9] plot ({3.35 	 	},       	{\x-1.4});
\draw [dashed,color=Emerald,domain=.5:1.9] plot ({2.39 	 	},       	{\x-1.4});
\filldraw[thick,RedOrange] (2.4,-.77) -- (2.4,-.27)  node[anchor=west] {$Q$}--(2.4,.23); 

\fill[opacity=.4,color=Cerulean] (2,-1.14)  -- (4,-1.14) -- (4,-.14) -- (2,-.14) -- cycle; 
\fill[opacity=.7,color=black!50!Cerulean] (.7,.601)  -- (2,-.14)  -- (2,-1.14) -- (.7,-0.399)-- cycle; 
\fill[opacity=.3,color=white!40!Cerulean] (.7,.601)  -- (2.7,.601)  -- (4,-.14)  -- (2,-.14)-- cycle; 
\draw (4.3,-1)  node[anchor= west] {\color{Cerulean!70!black}{$\omega_{Q}$}}; 
\end{tikzpicture}
\qquad
\begin{tikzpicture}
\draw[ dashed,color=black,xshift=2cm, yshift=-1.14cm,xstep=1cm] (-.2,-.2) grid (2.2,2.2);
\draw [dashed,color=black,domain=.5:2.2] plot ({\x 	 	},       	{-.57*\x + 2});
\draw [dashed,color=black,domain=.5:2.2] plot ({\x     +   1	},       	{-.57*\x + 2});
\draw [dashed,color=black,domain=.5:2.2] plot ({\x     +   2 	},       	{-.57*\x + 2});
\draw [dashed,color=black,domain=.5:2.2] plot ({\x 	  	},       	{-.57*\x + 1});
\draw [dashed,color=black,domain=.5:2.2] plot ({\x 	  	},       	{-.57*\x 	});
\draw [dashed,color=black,domain=.5:2.8] plot ({0.7 	 	},       	{\x-1});
\draw [dashed,color=black,domain=.5:2.8] plot ({1.35 	 	},       	{\x-1.4});
\draw [dashed,color=black,domain=.5:2.9 ] plot ({\x	+.1 	},       	{1.60});
\draw [dashed,color=black,domain=.5:2.9] plot ({\x 	+.7 	},       	{1.23});

\draw [dashed,color=Emerald,domain=.5:2.9 ] plot ({\x	+.7 	},       	{.23});
\draw [dashed,color=Emerald,domain=.5:2.2] plot ({\x 	+1  	},       	{-.57*\x + 1});
\draw [dashed,color=Emerald,domain=.5:2.9 ] plot ({\x	+.7 	},       	{-.77});
\draw [dashed,color=Emerald,domain=.5:2.2] plot ({\x 	 +1	},       	{-.57*\x });
\draw [dashed,color=Emerald,domain=.5:2.2] plot ({\x 	+2  	},       	{-.57*\x + 1});
\draw [dashed,color=Emerald,domain=.5:2.2] plot ({\x 	+2  	},       	{-.57*\x });
\draw [dashed,color=Emerald,domain=.5:2.8] plot ({1.7 	 	},       	{\x-1});
\draw [dashed,color=Emerald,domain=.5:2.8] plot ({2.7 	 	},       	{\x-1});
\draw [dashed,color=Emerald,domain=.5:2.9 ] plot ({\x	+.1 	},       	{.60});
\draw [dashed,color=Emerald,domain=.5:2.9 ] plot ({\x	+.1 	},       	{-.40});
\draw [dashed,color=Emerald,domain=.5:2.8] plot ({3.35 	 	},       	{\x-1.4});
\draw [dashed,color=Emerald,domain=.5:2.8] plot ({2.39 	 	},       	{\x-1.4});
\filldraw[RedOrange] (2.4,.23) circle (1.5pt) node[anchor=west] {$Q$}; 
\fill[opacity=.4,color=Cerulean] (2,-1.14)  -- (4,-1.14) -- (4,0.86) -- (2,0.86) -- cycle; 
\fill[opacity=.7,color=black!50!Cerulean] (.7,1.601)  -- (2,0.86)  -- (2,-1.14) -- (.7,-0.399)-- cycle; 
\fill[opacity=.3,color=white!40!Cerulean] (.7,1.601)  -- (2.7,1.601)  -- (4,0.86)  -- (2,0.86)-- cycle; 
\draw(4.3,-1)  node[anchor= west] {\color{Cerulean!70!black}{$\omega_{Q}$}}; 
\end{tikzpicture}\caption{In $\R^3$ we draw the set $\omega_Q$ associated with $Q$ in orange. We consider $Q\in \Q^{(2)}$ on the left, $Q\in \Q^{(1)}$ in the center and $Q\in \Q^{(0)}$ on the right.}\label{fig3}
\end{figure}

\begin{lemma}
\label{lem1}
Let $T\in \RR_k(\R^n)$ such that $\pt T\in \RR_{k-1}(\R^n)$ and $\MM(T)+\MM(\pt T)<\oo$. Let $Q\in Q^{(j)}$ for some $j\in \{k+1,\cdots,n\}$ and assume moreover that
\[
\supp T \subset \Sigma_Q.
\]
Then there exists ${\tilde T}\in \RR_k(\R^n)$ with $\pt {\tilde T}\in \RR_{k-1}(\R^n)$, there exist $U\in \RR_k(\R^n)$, $V\in \RR_{k+1}(\R^n)$ such that 
\be
\label{Id+supports}
T={\tilde T}+ U +\pt V, \qquad \supp U\, \cup\, \supp V\, \subset\, \ov Q,\qquad \supp {\tilde T}\subset \Sigma_Q\bks Q. 
\ee
Moreover, for any $\delta>0$, $\tilde T, \, U,\, V$ can be chosen in order to satisfy,
\begin{gather}
\MM_h({\tilde T}-T)\,  \leq\, c\, \MM_h(T\restr Q),\qquad\qquad  \quad\MM_h(V) \, \leq\, c\, \delta\, \MM_h(T\restr Q),\label{estimM1}\\
\MM_h(\pt {\tilde T}-\pt T)\, \leq\, c\, \MM_h(\pt T\restr Q),\qquad\qquad  \MM_h(U) \, \leq\, c\, \delta\,  \MM_h(\pt T\restr Q).\label{estimM2}
\end{gather}
where $c=c(n)>0$ is a constant. In addition, 
\begin{align}
&\mbox{if $\pt T\restr  Q$ is a polyhedral current then $U$ is a polyhedral current,}\label{ptTrestQ} \\
&\mbox{if  $T\restr  Q$ is a polyhedral current, so is $V$.}\label{TrestQ}
\end{align}
In the sequel, when applying the lemma, we choose $\tilde T$ satisfying the conclusions of the lemma and we note 
\be\label{PiQ}
\Pi_Q(T):=\tilde T.
\ee
\end{lemma}
The lemma and its proof follow the same lines as the deformation theorem of Federer and Fleming~\cite{FedFlem60} --- see~\cite{federer,KP2008}. However, there are two specific aspects in the present approach:
\begin{itemize}
\item The first lies in the presentation:  in the proof  of the original result, the authors project first $T\restr \ov Q$ on $\pt Q$ for every $Q\in\Q^{(n)}$, then they project the resulting current on $\pt Q$ for every $Q \in \Q^{(n-1)}$ and so forth, for $j=n,n-1,\cdots,k+1$. Here, we highlight the elementar operation of deforming the current in a single face $\ov Q$. This allows us to apply the deformation locally (in $\cup \ov Q$ where $Q$ ranges over a finite subset of $\Q^{(n)}$) and get some flexibility: we can use different grids in different regions. We could have obtained this flexibility by extending the local grids to a uniformly regular mesh defined in the whole space. Such delicate construction has been performed in~\cite{Feuvrier2012}.
\item In the original paper, the consecutive deformations are made of central projections of $T\restr Q$ from the center of $Q$ onto $\pt Q$. If the density of $\|T\|$ near the center is large, the projection may increase dramatically the mass of the current. To fix this, the original method is to translate the grid (the projection behaves well in average). Here, we insist in projecting on $\pt Q$ and not on one of its translates because the $k$-skeleton of $\pt Q$ contains a substantial part of the $h$-mass of $T$ that we cannot afford to increase in the deformation process. Instead of translating the grids, we move the center of projection in $\frac12Q$ to find a projection of $\|T\|$ on $\pt Q$ with good estimates. This is the method of {\it e.g.}~\cite[Sec. 5.1.1]{GMS1}.
\end{itemize}

\subsection{Proof of Lemma~\ref{lem1}}
Let $T\in \RR_k(\R^n)$, $j\in\{k+1,\cdots,n\}$ and $Q\in \Q^{(j)}$
satisfying the assumptions of Lemma~\ref{lem1}. Using a dilation and an affine isometry, we assume $\delta=1$, that $(e_1,\cdots,e_n)$ is the canonical basis of $\R^n$ and that $Q$ is centered at $0$.\medskip

\noindent
{\it Step 1.} Let us first select a good point for the projection of $T\restr Q$ and $\pt T\restr Q$ on $\pt Q$. We note $T=\tau(M,\theta,\xi)\in \RR_k(\R^n)$, $\pt T=\tau(M',\theta',\xi')\in \RR_{k-1}(\R^n)$.  For $a\in \frac12Q$ we consider the integrals
\[
I_h(a):=\int_{Q\cap M} \dfrac1{|y-a|^k} h(\theta(y))\, d\H^k(y), \qquad J_h(a):=\int_{Q\cap M'} \dfrac1{|y-a|^{k-1}}  h(\theta'(y))\, d\H^k(y).
\]
Integrating over $a\in \frac12Q$ and using Fubini, we compute
\[
\int_{\frac12Q} I_h(a)\, d\H^j(a)\, =\, \int_{Q\cap M} \lt( \int_{\frac12Q} \dfrac1{|y-a|^k}\,d\H^j(a) \rt) h(\theta(y))\, d\H^k(y).
\]
Using the change of variable $z=y-a$ in the inner integral and the fact that $y-\frac12Q\subset B_{2\sqrt{j}}$ for $y\in Q$, we obtain
\[
\int_{\frac12Q} I_h(a)\, d\H^j(a)\
\leq\, \lt(\int_{\R^j \cap B_{2\sqrt{j}}} \dfrac1{|z|^k}\,d\H^j(z) \rt)\int_{Q\cap M} h\circ \theta\, d\H^k\,
=\, \lt(\int_{\R^j \cap B_{2\sqrt{j}}} \dfrac1{|z|^k}\,d\H^j(z) \rt)\MM_h(T\restr Q).
\]
Since $k<j\leq n$, the first integral in the right hand side is finite and bounded by some constant only depending on $n$. We then have 
\[
 \int_{\frac12Q} I_h(a)\, d\H^j(a)\ \leq C\, \MM_h(T\restr Q).
\]
Similarly, 
\[
 \int_{\frac12Q} J_h(a)\, d\H^j(a)\ \leq C\, \MM_h(\pt T\restr Q).
\]
By Markov inequality, we deduce that there exists $a\in \frac12Q$ and a constant only depending on $n$ such that 
\be
\label{controlIa1}
I_h(a)\leq C\, \MM_h(T\restr Q),\qquad J_h(a)\,\ \leq C\, \MM_h(\pt T\restr Q).
\ee
\noindent
{\it Step 2.} We introduce a family of proper Lipschitz mappings $u_\eps :\R^n\to\R^n$. First for $y\in Q\bks \{a\}$ we define $u(y)$ as the radial projection of $y$ on $\pt Q$ with respect to $a$. Next, for $\eps\in(0,1/2)$ and $y\in \Sigma_Q$ we define, 
\[
u_\eps(y):=\lt\{\begin{array}{cl}
({|y-a|}/\eps) u(y)+\lt(1-{|y-a|}/\eps\rt) y &\mbox{ if } y \in Q\cap B_{\eps}(a), \\
u(y) &\mbox{ if } y \in Q\bks B_{\eps}(a), \\
y & \mbox{ if } y \in \Sigma_Q\bks Q.
\end{array}\rt.
\]
The mapping $u_\eps$ is well defined and Lipschitz on $\Sigma_Q$. We extend it on $\om_Q$ to obtain a Lipschitz mapping on $\R^n$, still noted $u_\eps$. Notice that we have $u_\eps(y)\to u(y)$ as $\eps\dw0$ locally uniformly in $\ov Q\bks \{a\}$.

Next, we define $z_\eps:[0,1]\times \R^n\to \R^n$ as $z_\eps(t,y) = (1-t) y + t u_\eps(y)$. The homotopy formula~\eqref{homform} leads to  
\be\label{homformeps}
\tilde T_\eps - T= -\pt V_\eps - U_\eps,
\ee
with $T_\eps, U_\eps\in \RR_k(\R^n)$ and $V_\eps\in \RR_{k+1}(\R^n)$ defined as 
\[
\tilde T_\eps :=u_\eps \pf T,\qquad V_\eps:=- z_\eps\pf\lt(\llb(0,1)\rrb\times T\rt) ,\qquad U_\eps:=- z_\eps\pf \lt(\llb(0,1)\rrb\times \pt T\rt). 
\]
We notice that $z_\eps(t,y)$ does not depend on $t$ on $[0,1]\times (\Sigma_Q \bks Q)$, so
\[
z_\eps\pf \lt(\llb(0,1)\rrb \times T\!\restr  (\Sigma_Q \bks Q)\rt)=0,\qquad z_\eps\pf \lt(\llb(0,1)\rrb \times \pt T \!\restr (\Sigma_Q \bks Q)\rt)=0,
\] 
and since by assumption $\supp T\subset \Sigma_Q$, we can write
\[
V_\eps:=- z_\eps\pf(\llb(0,1)\rrb\times (T\restr Q)) ,\qquad U_\eps:=- z_\eps\pf(\llb(0,1)\rrb\times (\pt T\restr Q)). 
\]
Similarly, since $u_\eps\equiv \Id$ on $\Sigma_Q\bks Q$, we also have,
\[
\tilde T_\eps - T=u_\eps \pf (T\restr Q) - T\restr Q. 
\]
{\it Step 3.} We wish to send $\eps$ towards 0 in~\eqref{homformeps}. For this we notice that for $0<\eps'<\eps<1/2$ and $y\in Q$, we have $|Du_\eps(y)|\leq C/|y-a|$ and $\supp (u_\eps-u_\eps')\subset \ov B_{\eps}(a)$. By~\eqref{estimPFLip}, we deduce 
\be\label{MMwueps}
\MM_h(u_\eps \pf T\restr Q)\, \leq \, C I_h(a)\stackrel{\eqref{controlIa1}}\leq C \MM_h(T\restr Q).
\ee
Moreover, 
\[
\MM_h(u_\eps \pf T\restr Q-u_{\eps'} \pf T\restr Q)\,\leq\, C \int_{Q\cap M\cap B_{\max(\eps,\eps')}(a)} \dfrac1{|y-a|^k} h(\theta(y))\, d\H^k(y)\ \stackrel{\eps',\eps\dw0}\longto \ 0
\]
Since $\MM\leq (1/\a)\MM_h$ (recall~\eqref{condf2}), we see that the family $\{\tilde T_\eps\}$ has the Cauchy property for the $\MM$-distance. Passing to the limit we have $\tilde T_\eps\to \tilde T$ as $\eps\dw 0$ and moreover,~\eqref{MMwueps} yields the first part of~\eqref{estimM1}. Similarly, we deduce from~\eqref{controlIa1} that $U_\eps$ and $V_\eps$ have limits noted $U$ and $V$ as $\eps\dw0$ that satisfy~\eqref{estimM1}\eqref{estimM2}. Passing to the limit in~\eqref{homformeps} we have the desired  decomposition $T=\tilde T + U +\pt V$ and from the properties of the support of $\tilde T_\eps, U_\eps$ and $V_\eps$ we have $\supp ({\tilde T}-T)\,\cup\, \supp U\, \cup\, \supp V\, \subset\, \ov Q$. We also have to check the last inclusion of~\eqref{Id+supports}. From the definition of $u_\eps$, we have 
\[
\MM_h(\tilde T_\eps \restr Q) =\MM_h(u_\eps\pf [T\restr B_\eps(a)])\, \leq\, C  \int_{Q\cap M\cap B_{\eps}(a)} \dfrac1{|y-a|^k} |h(\theta(y))|\, d\H^k(y)\ \stackrel{\eps\dw0}\longto \ 0.
\]
We deduce that $\MM_h(\tilde T\restr Q)=0$ and since $\supp \tilde T\subset \Sigma_T$, we conclude that $\supp \tilde T\subset \Sigma_T\bks Q$. This proves~\eqref{Id+supports}.  \medskip 

\noindent
{\it Step 4.} Eventually, let us assume that $T\restr Q$ is a polyhedral current. If $a\in \supp T$, then there exist constants $c,\eta>0$ and a non empty open polyhedral cone $\mc{C}$ with vertex $a$ and dimension $k$ such that $\|T\|\geq c\H^k\restr {\mc{C}\cap B_{\eta}(a)}$. This implies $I_h(a)=+\oo$ and contradicts our choice for $a$. Hence $d(a,\supp T)>0$ and for $0<\eps<d(a,\supp T)$, $z_\eps$ does not depend on $\eps$ on $[0,1]\times \supp T$. For such $\eps$, we have,
\[
V_\eps=- z_\eps\pf(\llb(0,1)\rrb\times (T\restr Q)) = V.
\]
From the explicit form of $u$ we see that $V$ is a polyhedral current. Indeed, the polyhedral current $T\restr Q$ can be decomposed as a linear combination of closed convex oriented $k$-polyhedrons $T_S=\tau(S,\xi,1)$ with $a\not \in S$ and $u(S)\subset L$ for some  $(j-1)$ face $L$. Then, for $\eps>0$ small enough, 
\[
z_\eps \pf(\llb(0,1)\rrb\times (T_S))= \tau(\tilde S,\tilde \xi,1),
\]
 where $\tilde S$ is the convex hull of $S\cup u(S)$ and $\tilde \xi:=|\zeta|^{-1}\zeta$, $\zeta:=(u(y)-y)\wedge \xi$ for some $y\in S$.

 Similarly, if $\pt T\restr Q$ is a polyhedral current then $U$ is a polyhedral current. This ends the proof of Lemma~\ref{lem1}.

\section{Proof of Theorem~\ref{thm1}}
\label{S4}
Before coming to the proof we set some notation and state a covering lemma.
\subsection{Notation for closed $k$-cubes and a covering lemma}
\label{S41}
Given  $x\in \R^n$, $\ell>0$ and $e^\pl=\{e_1,\cdots,e_k\}\subset \R^n$ an orthonormal family, we note $F=F_{x,\ell,e^\pl}$ the $k$-dimensional closed cube centered in $x$
\[
F= x+ \lt\{\sum_{j=1}^k t_j e_j : -\ell/2\leq t_j \leq \ell/2\mbox{ for }j\in \{1,\cdots,k\}\rt\}.
\] 

Conversely, given the $k$-cube $F$, we note $x_F=x$, $\ell_F=\ell$, $e_F=e$.  For $\lambda>0$ we note $\lambda F=F_{x,\lambda\ell,e}$ the cube with same center and orientation as $F$ and with side length $\lambda\ell_F$. To each $k$-cube $F$, we associate a family $e^\perp_F=\{e_{k+1},\cdots,e_n\}$ so that $(e_1,\cdots,e_n)$ form an orthonormal basis. For $\delta>0$, we define the closed $n$-dimensional set $F^\delta:= F+F_{0,\delta,e^\perp_F}$. 
Equivalently, 
\[
F^\delta =x+\lt\{\sum_{j=1}^n t_j e_j: |t_j| \leq \ell/2\mbox{ for }j\in \{1,\cdots,k\}, \, |t_j |\leq \delta/2 \mbox{ for }j\in \{k+1,\cdots,n\}\rt\}.
\]

In the sequel we deal with coverings by cubes with possibly different orientations. For this we need Morse's version of the Besicovitch covering lemma~\cite{Morse47}. Actually, we use a corollary of the (Morse)-Besicovitch covering lemma (see~\cite[Theorem 2.19]{Am_Fu_Pal}).
\begin{lemma}[Morse-Vitali-Besicovitch covering]
\label{lemMB}
Let $\mu$ be a positive Radon measure over $\R^N$ and let $A\subset \R^n$ such that $\mu(\R^N\bks A)=0$. For every $x\in A$, let $\mc{F}_x$ be a family of closed subsets of $\R^n$ that contain $x$ and note $\mc{F}:=\cup_x\mc{F}_x$.\\
We assume that $\mc{F}$ is a fine covering of $A$, that is, for every $x\in A$ and for every $\rho>0$, 
\[\{G \in \mc{F}_x :\diam G<\rho \}\neq\void.\]
We also assume that $\mc{F}$ satisfy a $\lambda$-Morse property: there exists $\lambda>0$ such that for every $x\in A$ and every $F\in \mc{F}_x$,
\[
B_\rho(x)\subset F\subset B_{\lambda\rho}(x),\quad \mbox{for some $\rho>0$ with moreover $F$ star-shaped with respect to $B_\rho(x)$.}
\]
Then, for every $\eps>0$,  there exists a finite subset $\mc{F}_\eps\subset \mc{F}$ such that the elements of $\mc{F}_\eps$ are disjoint and $\mu(\R^N\bks \bigcup\mc{F}_\eps)<\eps$.
\end{lemma} 

\subsection{Pushing forward (most of) $T$ on $k$-cubes}
\label{S42}
Let $R=\tau(M,\xi,\theta)\in \RR_k(\R^n)$ with $\pt R\in \PP_{k-1}(\R^n)$ and $\MM_h(R)<\oo$ as in the statement of Theorem~\ref{thm1}.
We first show that we can assume that most of the $h$-mass of $R$ lies on a finite set of disjoint $k$-cubes. 
\begin{lemma}
\label{lemS42}
For every $\eta>0$, there exists a compact set $K$ which is a finite union of disjoint closed $k$-cubes, with $K\cap\, \supp \pt R=\void$, there exist $\tilde R\in \RR_k(\R^n)$ and $\tilde V\in \RR_{k+1}(\R^n)$ such that $R=\tilde R+\pt \tilde V$ with 
\[
\MM_h(\tilde V)<\eta,\qquad \MM_h(\tilde R)<\MM_h(R)+\eta,\qquad  \MM_h(\tilde R \restr [\R^n\bks  K])<\eta,\qquad\supp \tilde V\subset \supp R +B_{\eta}.
\]
\end{lemma}
\begin{proof}~Let $\eps\in(0,1/2)$ be a small parameter that will be fixed at the end of the proof.\smallskip\\
{\it Step 1. Most of the $h$-mass of $T$ lies on a finite union of $C^1$-graphs over small $k$-cubes.}\\
 Since $R=\tau(M,\xi,\theta)$ is rectifiable with $\MM_h(R)<\oo$, the measure
\[
\mu:=h(\theta)\,\H^k\restr M
\]
is a rectifiable measure and there exists a compact, orientable, $k$-manifold ${\cal N}$ of class $C^1$ with
\[
\mu (\R^n\bks {\cal N})\,<\, \eps.
\]
Moreover, since $\pt R$ is a $(k-1)$-polyhedral current, we have $\mu(\supp \pt R+B_\rho)\to 0$ as $\rho\dw0$. Removing from ${\cal N}$ (if necessary) a small neighbourhood of $\pt R$ we can assume 
\be
\label{loindubord}
d({\cal N},\supp \pt R)>0. 
\ee
Extending $\theta$ by $0$ on ${\cal N}\bks M$, we have $h\circ\theta\rest{{\cal N}}\in L^1({\cal N},\H^k)$ and $\H^k({\cal N}\bks A)=0$ where $A\subset {\cal N}$ denotes the set of Lebesgue points of the mapping $h\circ\theta\rest{{\cal N}}$. In particular:
\begin{property}\label{property}
 for every $x\in A$ there exists $\delta_x>0$ with (recall~\eqref{loindubord}) 
\be
\label{ldb2}
\delta_x<\eps,\qquad  \sqrt{k} \delta_x/2<d(x,\supp \pt T)
\ee 
such that: for all $k$-cube $F$ tangent to ${\cal N}$ at $x$ with $x_F=x$ and side length $\ell_F\leq \delta_x$, there hold: 
\begin{enumerate}
\item  $F^{\ell_F}\cap {\cal N}$ is the graph of a $C^1$ function $g_F :F\to \vect e_F^\perp$ such that $g_F(x)=0$ and $\|Dg_F\|_\oo<\eps/\sqrt{k}$. In particular, $F^{\ell_F}\cap {\cal N}\subset F^{\eps\ell_F}$.
\item  Noting $\mc{G}_L:=\{y+g_F(y):y\in L\}$ the graph of $g_F\rest{L}$ for $L\subset F$, we have 
\be
\label{q}
\MM_h\lt(T\restr  \mc{G}_{F\bks(1-\eps)F}\rt)\ <\ 2k\eps\,\MM_h(T\restr \mc{G}_F).
\ee
\end{enumerate}
\end{property}
The first point comes from the $C^1$ regularity of ${\cal N}$. The second point is a consequence of the fact that $x$ is a Lebesgue point of $h\circ \theta\rest{{\cal N}}$. Indeed, using the parameterization $y\in F\mapsto y+g_F(y)$ of ${\cal N}\cap F^{\ell_F}_x$,~\eqref{q} rewrites as
\[
\label{q2}
\int_{F\bks(1-\eps)F}f(y)\, dy \leq\ 2k\eps\,\int_F  f(y)\, dy,
\]
with $f(y):=h( \theta(y+g_F(y)))\sqrt{1+|Dg_F(y)|^2}$. Since $x_F$ is a Lebesgue point of $f$ and $Dg_F(y)\to0$ as $F\ni y\to x_F$, this inequality holds true for $\delta_x$ small enough.\medskip\\

Let us call $\mc{F} $ the family of the closed $n$-cubes $F^{\ell_F}$ with $x_F\in A$ and $\ell_F < \delta_{x_F}$. These cubes are convex and satisfy the Morse condition, indeed
\[
\ov B_{\ell_F/2}(x_F)\subset F^{\ell_F} \subset \ov B_{\sqrt{n}\ell_F/2}(x_F).
\] 
Moreover, given such $F^{\ell_F}\in \mc{F}$, we have $\lambda  F^{\ell_F}\in \mc{F}$ for  $0<\lambda<1$ and  
the family $\mc{F}$ is a fine cover of $A$. Applying the  Morse-Besicovitch covering lemma~\ref{lemMB} to the measure $\mu\restr A$, there exists a finite subset $\mc{F}_\eps$ of $\mc{F}$ such that the elements of $\mc{F}_\eps$ are disjoint and 
\be\label{estimOnDc}
\mu(A\bks D) <\eps\qquad\mbox{with}\quad D:=\bigcup_{F^{\ell_F}\in\mc{F}_\eps}F^{\ell_F}.
\ee
Moreover, removing the  elements $F^{\ell_F}$ such that $\mu(F^{\ell_F})=0$ we can assume that 
\be\label{Feps_support}
F^{\ell_F}\cap \supp R\neq \void\quad\mbox{for every }F^{\ell_F}\in \mc{F}_\eps.
\ee
{\it Step 2. Pushing the graphs of $g_F$ onto the $k$-cubes $F$.}\\

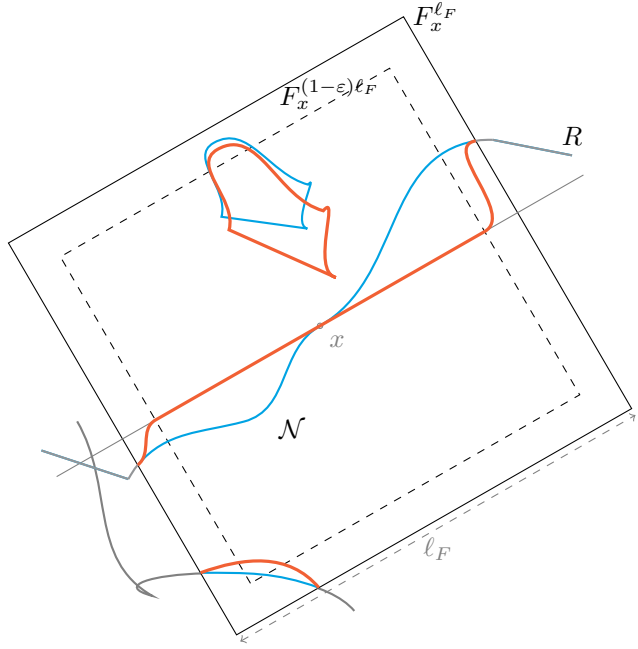
\begin{figure}
\centering
\begin{tikzpicture}[rotate=30]
\draw[thick,gray]  (-1.5,-3.5) to[out=100] (-3.7,-2) to[out=120,in=270](-3.4,.5);
\fill[white] (-3,-3)rectangle(-1,-2);
\draw[dashed] (-2.5,-2.5) rectangle (2.5,2.5)node[anchor=  north east]{$F^{(1-\eps)\ell_F}_x$};;
\draw (-3,-3) rectangle (3,3)node[anchor= west]{$F^{\ell_F}_x$};
\draw (-1,-1) node {$\mathcal{N}$};
\draw[gray] (-4,0)--(0,0)circle (1pt) node[anchor=north west]{$x$}--(4,0);
\draw[gray,dashed,<->] (-3,-3.1)--  (0,-3.1) node[anchor=north]{$\ell_F$}--(3,-3.1);

\draw[thick,Cerulean]  (.5,1.2) -- (-.4,1.9) to[in=180](0,2.8) to[out=0](.8,1.7) to[out=20](.5,1.2);
\draw[thick,Cerulean]  (-4,.4)--(-3.2,-.5) (3.2,1)--(4,.3) node [anchor = south] {\color{black}{$R$}};
\draw[thick,Cerulean]  (-3,-.4) to[out=20,in=165] (-1.5,-.6) to [out=345,in=180]  (0,0) to[out=0,in=160]  (3,1.1);

\draw[thick, Cerulean]  	(-1.75,-3) to[in=-30,out=130] (-3,-2.05);

\draw[thick,Gray]  (-4,.4)--(-3.2,-.5) (3.2,1)--(4,.3) ;
\draw[thick,Gray]  (-3.2,-.5) to[out=30,in=200](-3,-.4)  (3,1.1) to[out=-20,in=150] (3.2,1)  ;

\draw[very thick, RedOrange] 	(.5,.45) -- (-.4,1.7) to[in=180](0,2.7) to[out=0](.8,1.3) to[out=20](.5,.45);
\draw[very thick, RedOrange]  	(-1.75,-3) to[in=-10,out=100] (-3,-2.05);
\draw[very thick, RedOrange]  (-3,-.4) to[out=8,in=180]  (-2.5,0)--(2.5,0) to[out=0,in=160] (3,1.1) ;
\end{tikzpicture}
\label{fig4}\caption{Representation of the projection step described in \textit{Step 2}. In blue and gray the original support of the rectifiable current $R$, in orange and gray the deformed one.}
\end{figure}
Let us introduce a smooth cut-off function $\chi:[-1/2,1/2]^n\to[0,1]$ such that $\chi\equiv 1$ on $(1-\eps)[-1/2,1/2]^n$, $\chi\equiv 0$ on the boundary $\de [-1/2,1/2]^n$ and $\|D\chi\|_\oo\leq 4/\eps$.\\
Let $F^{\ell_F}\in \mc{F}_\eps$ and let $F$ be the associated closed $k$ cube tangent to ${\cal N}$ at its center. Up to a change of frame, we assume $x_F=0$ and $F=[-\ell_F/2,\ell_F/2]^k\times \lt\{0_{\R^{n-k}}\rt\}$, so that $F^{\ell_F}=[-\ell_F/2,\ell_F/2]^n$. \\
For $y\in \R^n$, we write $y=(y^\pl,y^\perp)$ its decomposition in $\R^k\times\R^{n-k}$. With this notation we define the diffeomorphism $u_F:\R^n\to\R^n$ as  
\[
u_F(y):=
\begin{cases}
\lt(y^\pl,y^\perp -\chi(y/\ell_F)g_F(y^\pl)\rt)& \mbox{if } y\in  F^{\ell_F},\\
\qquad y & \mbox{if }y\not\in  F^{\ell_F}.
\end{cases}
\]
This mapping is Lipschitz with $\|Du_F\|_\oo\leq C$ (notice that from the first point of Property~\ref{property}, we have $\|g_F\|_\oo\leq \eps\ell_F$).
 We set, 
\[
\what R_F:=u_F\pf R.
\]
Since $u_F=\Id$ in $ \ov{(F^{\ell_F}_x)^c}$, we have  $\supp(\what R_F-R)\subset   F^{\ell_F}_x$ and by~\eqref{estimPFLip}, 
\be\label{EstimhatTFc}
\MM_h(\what R_F\restr[  F^{\ell_F}_x\bks F])\, \leq\, C \MM_h(R\restr[ F^{\ell_F}_x\bks \mc{G}_F]).
\ee
Taking into account $D[\chi(y/\ell_F)]\equiv 0$ on $(1-\eps)F$ and~\eqref{q}, we also have
\be\label{EstimhatTrestF}
\MM_h(\what R_F\restr F)\, \leq\, (1+\eps) \MM_h(R\restr [{\cal N}\cap  F^{\ell_F}_x]).
\ee

We also define $z_F:(t,y)\in[0,1]\times\R^n\mapsto t u_F(y)+(1-t) y\in \R^n$. By~\eqref{ldb2}, $\pt R\restr  F^{\ell_F}_x=0$ and since $z_F(t,y)=y$ in $[0,1]\times ( F^{\ell_F})^c$, the homotopy formula~\eqref{homform} reduces to 
\[
R\, =\, \what R_F + \pt \what V_F,\qquad\mbox{with }\what V_F:=z_F\pf(\llb(0,1)\rrb \times T).
\]
By construction, $\supp \what V_F\subset  F^{\ell_F}_x$ and from~\eqref{estimPFLip}, we have
\[
\MM_h(\what V_F)\, \leq \, \|Dz_F\|_\oo^{k+1} \MM_h\lt(R\restr \lt[  F^{\ell_F}_x\bks \mc{G}_{(1-\eps)F}\rt] \rt) + \|Dz\|_{L^\oo(\mc{G}_{(1-\eps)F})}^{k+1} \MM_h\lt(R\restr \mc{G}_{(1-\eps)F}\rt).
\]
Since $\|Dz_F\|_\oo\leq C$ and $\|Dz\|_{L^\oo(\mc{G}_{(1-\eps)F})}\leq C\eps$, this leads to
\be\label{EstimhatR}
\MM_h(\what V_F)\,\leq\, C\, \lt\{ \MM_h(R\restr [ F^{\ell_F}_x\bks {\cal N}])+ \eps\, \MM_h(R\restr [ F^{\ell_F}_x\cap {\cal N}])\rt\}. 
\ee
Repeating the construction for $F\in\mc{F}_\eps$, we obtain $R=\tilde R +\pt \tilde V$. The estimates~\eqref{EstimhatTFc}, \eqref{EstimhatTrestF} and \eqref{EstimhatR} lead to,
\[
\MM_h(\tilde R\restr (\Int K)^c),\quad \MM_h(\tilde R) -\MM_h(R),\quad \MM_h(\tilde  V) \,\leq\, C\lt\{ \MM_h(R\restr {\cal N}^c)+\eps\MM_h(R\restr {\cal N})\rt\}.
\]
Using~\eqref{estimOnDc} to estimate the first term in the right hand side, we obtain, 
\[
\MM_h(\tilde R\restr (\Int K)^c),\quad \MM_h(\tilde R) -\MM_h(R),\quad \MM_h(\tilde  V) \,\leq\,C(1+\MM_h(R))\eps.
\] 
Eventually, by construction $\supp \tilde V \subset \cup \{\ov{ F^{\ell_F}_x}: F^{\ell_F}_x\in {\cal F}_\eps\}$ and by~\eqref{Feps_support} this leads to $\supp \tilde V \subset \supp R+B_{2\sqrt{n}\eps}$. Choosing $\eps>0$ small enough, the lemma is proved.
\end{proof}

\subsection{Cleaning the neighborhood of the $k$-cubes of $K$}
\label{Sproof}
By Lemma~\ref{lemS42}, we can now assume that there exists a finite union  of closed disjoint  $k$-cubes, $K=F_1\cup \cdots\cup F_m$ such that  
\be
\label{assumption-Tmainlyonkcubes}
\MM_h(R\restr K^c)<\eta,
\ee
and $K\cap \supp \pt R=\void$. Using the notation of Section~\ref{S41}, there exist a positive integer $N$ such that noting $\delta_j:= \ell_{F_j}/ N$, the boxes $F_1^{\delta_1},\cdots,F_m^{\delta_m}$ are disjoints, 
\[
\max_{1\leq l\leq m} \delta_l<\eta,
\]
and $\tilde K$ does not intersect $\pt R$  where we define
\be\label{tK}
\tilde K:=F_1^{\delta _1}\cup \cdots\cup F_m^{\delta _m}.
\ee 

Up to dilations and displacements, all these polyhedrons are of the form  
\[
F^*:=[-N,N]^k\times [-1,1]^{n-k}.
\]
More precisely, $F_l^{\delta_l}=\psi_l(F^*)$ with $\psi_l:\R^n\to\R^n$ diffeomorphism such that $D\psi_l\equiv \lambda_l O_l$ for some $\lambda_l>0,\ O_l\in \SO_n(\R)$.\\
Let $Q_0=(0,1)^n$ and for $j=0,\cdots, n$, let $\Q^{(j)}$  be the $j$ skeleton associated with the partition of $\R^n$ based on translates of $Q_0$ (this is the same notation as in the beginning of Section~\ref{S3}).  We note 
\[
\Q_{F^*}^{(j)}:=\lt\{Q\in \Q^{(j)} : Q\subset F^*\rt\}\ \mbox{ for $j=0,\cdots, n$\quad and }\ \Q_{F^*}:=\bigcup_{j=0}^n \Q_{F^*}^{(j)}.
\]
The elements of $\Q_{F^*}$ form a partition of $F^*$. Now, applying the mapping $\psi_l$, we obtain a similar decomposition of $F_l^{\delta_l}$, that is, noting  $\Q^{(j)}_l:=\{ \psi_l(Q):Q\in \Q_{F_1}^{(j)} \}$, $\Q_l:=\cup_j \Q^{(j)}_l$, $\Q_l$ is a partition of $F_l^{\delta_l}$. The elements of $\Q_l^{(j)}$ are open $j$-cubes with side length $\delta_l$ and for $0\leq j <n$, $Q^{(j)}_l$ is formed by the faces of the elements of $\Q^{(j+1)}_l$. Before applying the deformation lemma, let us introduce some notation 
\begin{eqnarray*}
\tilde\Q_l^{(j)}:=\{Q\in \Q_l^{(j)} : \ Q\subset \inter (F_l^{\delta_l})\}.
\end{eqnarray*}
Recalling the characterization~\eqref{caractomQ} of $\om_Q$, we see that 
\[
\om_l^{(j)}:=\cup \{Q\in \tilde \Q_l^{(i)}: j+1\leq i\leq n \}=\cup\{\om_Q:Q\in \tilde \Q_l^{(j)}\}. 
\]
We note 
\[
\Sigma_{l}^{(j)}:=F_l^{\delta_l} \bks \om_l^{(j)}=\pt F_l^{\delta_l}\cup\lt[\cup \{Q\in \tilde \Q_l^{(i)}:0\leq i \leq j\}\rt].
\]
We notice that
\be\label{Sigmalj1}
\Sigma_{l}^{(j)}\,\subset\, \Sigma_Q\cap F_l^{\delta_l}\qquad\mbox{ for every }Q\in \tilde \Q_l^{(j)}.
\ee
In the sequel, we apply successively the local deformation lemma to $T=R$ with respect to all the elements of $\tilde\Q_1^{(n)}$. We obtain a new current $ R^{n-1}$ which is supported on $\lt[\supp R^n \bks F_1^{\delta_1} \rt]\cup \Sigma_{1}^{(n-1)}$. By~\eqref{Sigmalj1}, this condition allows us to apply the local deformation lemma to $ R^{n-1}$ with respect to all the elements of  $\tilde\Q_1^{(n-1)}$. We then continue the deformations with respect to the elements of $ \tilde \Q^{(n-2)}_1$, $ \tilde \Q^{(n-3)}_1,\cdots$, up to $\tilde\Q_1^{(k+1)}$. Let us give some details and state the estimates. \\ 
\begin{figure}[!h]
\centering
\begin{tikzpicture}[scale=1.3]

\draw[ultra thick,ForestGreen,dashed] (-3,-1) grid (3,1);
\draw[thick,Gray]  (-4,.4)--(-3.2,-.5) (3.2,1)--(4,.3) node [anchor = south] {\color{black}{$R$}};
\draw[thick,Gray] 	(.5,.45) -- (-.4,1.7) to[in=180](0,2.7) to[out=0](.8,1.3) to[out=20](.5,.45);
\draw[thick,Gray]   (-3.2,-.5) to[out=30,in=188] (-3,-.4) (3,1.1)to[out=-30,in=150] (3.2,1) ;
\draw[thick,Gray]  (-3,-.4) to[out=8,in=180]  (-2.5,0)--(2.5,0) to[out=0,in=160] (3,1.1) ;
\draw[thick, RedOrange] (2.5,0) --(3,0) -- (3,1) --(2.8,1);
\draw[thick, RedOrange] (.1,1) --(.67,1);
\draw[thick, RedOrange](-3,-.4) --(-3,0) --  (-2.5,0);
\begin{scope}
\clip (-3,-1) rectangle (3,1);

\draw[gray] (-4,0)--(0,0)circle (1pt) node[anchor=north west]{\color{black}$x$}--(4,0);
\draw (3,-1) node[anchor= north]{$[-N,N]\times[-1,1]$};
\draw(-2.5,.5) node {\small$Q_0$} (2.5,-.5) node {\small $Q_{M-1}$};
\draw[thick,Cerulean] 	(.5,.45) -- (-.4,1.7) to[in=180](0,2.7) to[out=0](.8,1.3) to[out=20](.5,.45);
\draw[thick,Cerulean]  (-3,-.4) to[out=8,in=180]  (-2.5,0)--(2.5,0) to[out=0,in=160] (3,1.1) ;
\filldraw[RedOrange] (-2.5,-.6) circle (1pt)node[anchor= north west] {$a$};
\filldraw[RedOrange] (.75,.6) circle (1pt)node[anchor= north ] {$a$};
\filldraw[RedOrange] (2.2,.4) circle (1pt)node[anchor= south west] {$a$};
\end{scope}

\end{tikzpicture}
\label{fig5}\caption{Example of application of Lemma~\ref{lem1} in each cube in $\tilde Q^{(2)}_1$ for an $R\in\RR_1(\R)$. In dashed green we represent the set $\Sigma^{(1)}_1$. In blue and gray the original support of the rectifiable current $R$, in orange and gray the deformed one}
\end{figure}
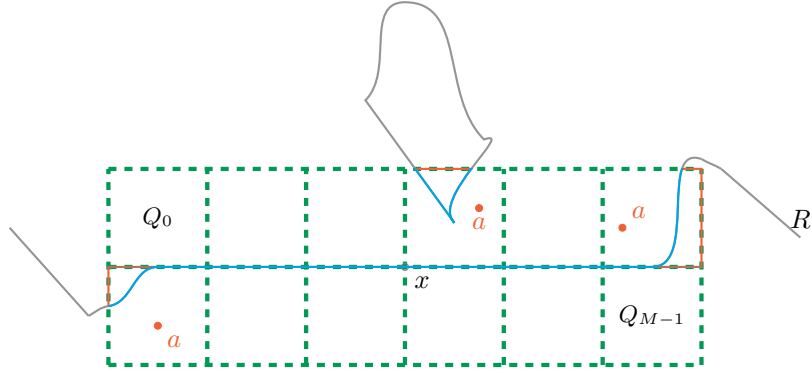
Let us number the cubes of $\tilde Q^{(n)}_1$ as $\{Q_0,\cdots,Q_{M-1}\}$. We apply successively Lemma~\ref{lem1} with $Q=Q_0,\, \cdots,\, Q=Q_{M-1}$. Using the notation~\eqref{PiQ}, we set
\[
R^{n-1}:=\Pi_{Q_{M-1}}\circ\cdots\circ\Pi_{Q_0}(R)
\]
From~\eqref{Id+supports} and using the fact that $\pt R \restr F_1^{\delta_1}=0$, we obtain the decomposition,
\be
\label{Id+supports-1}
R =  R^{n-1} + \pt  V^{n-1},\qquad \supp  V^{n-1}\, \subset\, F_1^{\delta_1}, \qquad \supp R^{n-1}\subset \lt[\supp R\bks F_1^{\delta_1}\rt]\cup \Sigma_1^{(n-1)},
\ee
with the estimates,
\begin{equation}\label{estimM1-1}
\begin{gathered}
\MM_h(R^{n-1}- R) \leq c\, \MM_h(R\restr (F_1^{\delta_1}\bks F_1)),\\
\MM_h(V^{n-1}) \leq c\,\eta\, \MM_h(R\restr(F_1^{\delta_1}\bks F_1)).
\end{gathered}
\end{equation}
If $k<n-1$, from the last property of \eqref{Id+supports-1} and \eqref{Sigmalj1}, we can apply successively Lemma~\ref{lem1} to $R^{n-1}$ with $Q$ running  over the elements of $ \tilde \Q^{(n-1)}_1$. We obtain the decomposition 
\be
\label{Id+supports-2}
R^{n-1} =  R^{n-2} + \pt  V^{n-2},\qquad \supp  V^{n-2}\, \subset\, F_1^{\delta_1}, \qquad \supp R^{n-2}\subset \lt[\supp R\bks F_1^{\delta_1}\rt]\cup \Sigma_1^{(n-2)},
\ee
with estimates similar to those in~\eqref{estimM1-1} we obtain
\[
\begin{gathered}
\MM_h(R^{n-2}- R^{n-1}) \leq c\, \MM_h(R^{n-1}\restr (F_1^{\delta_1}\bks F_1)),\\
\MM_h(V^{n-2}) \leq c\,\eta\, \MM_h(R^{n-1}\restr(F_1^{\delta_1}\bks F_1)).
\end{gathered}
\]
Again, from the last property of~\eqref{Id+supports-2}, we see that  $\supp R^{n-2}\subset \Sigma_Q$ for every $Q\in \tilde \Q^{(n-2)}_1$ and, if $k<n-2$, we can proceed further applying the deformation lemma with respect to the elements of $\tilde\Q^{(n-2)}_1$. Continuing the argument up to $\tilde \Q^{(k+1)}_1$ and then repeating the construction in all the remaining boxes $F_2^{\delta_2},\cdots, F_m^{\delta_m}$ that form $\tilde K$ (recall the notation~\eqref{tK}), we obtain the decomposition 
\be
\label{Id+supports-3}
R = \tilde R + \pt \tilde V,\qquad \supp \tilde V\, \subset\,K, \qquad \supp {\tilde R}\subset \lt[\supp R\bks \tilde K\rt]\cup\lt[\bigcup_{l=1}^m \Sigma_l^{(k)}\rt].  
\ee
Returning to $F_l^{\delta_l}=\psi_l(F^*)$, we see that  $\Sigma_l^{(k)}=F_l$, so that 
\[
\supp {\tilde R}\subset \lt[\supp R\bks \Int \tilde K \rt]\cup K.  
\]
Moreover, by subadditivity,
\begin{equation}
\label{estimM1-2}
\begin{gathered}
\MM_h({\tilde R}-R) \leq C\, \MM_h(R\restr K^c)\ \stackrel{\eqref{assumption-Tmainlyonkcubes}}\leq\ C\eta,\\
\MM_h(\tilde V) \leq C\,\delta\, \MM_h(R\restr K^c)\ \stackrel{\eqref{assumption-Tmainlyonkcubes}}\leq\ C \delta \eta.
\end{gathered}
\end{equation}
Eventually, let $l\in\{1,\cdots,m\}$ and let $S:=\tilde R\restr \inter(F_l^{\delta_l})\in \RR_k(\R^n)$. By construction $S$ is supported in the $k$-skeleton 
\[
X:=\Sigma_l^{(k)}\setminus \pt F_l^{\delta_l}=\cup \{Q\in \tilde \Q_l^{(i)}:0\leq i \leq k\},
\]
and $\pt S$ is supported in $Y:=\ov X\cap \pt F_l^{\delta_l}$ which is a finite union of $(k-1)$ closed cubes.
By Lemma~\ref{lemConstancy} we conclude that $S$ is a polyhedral current. Therefore,
\be
\label{TildeTPolyhedral}
\tilde R\restr  \Int \tilde K \mbox{ is a $k$-polyhedral current}. 
\ee

\subsection{Deformation of the remaining parts and conclusion}
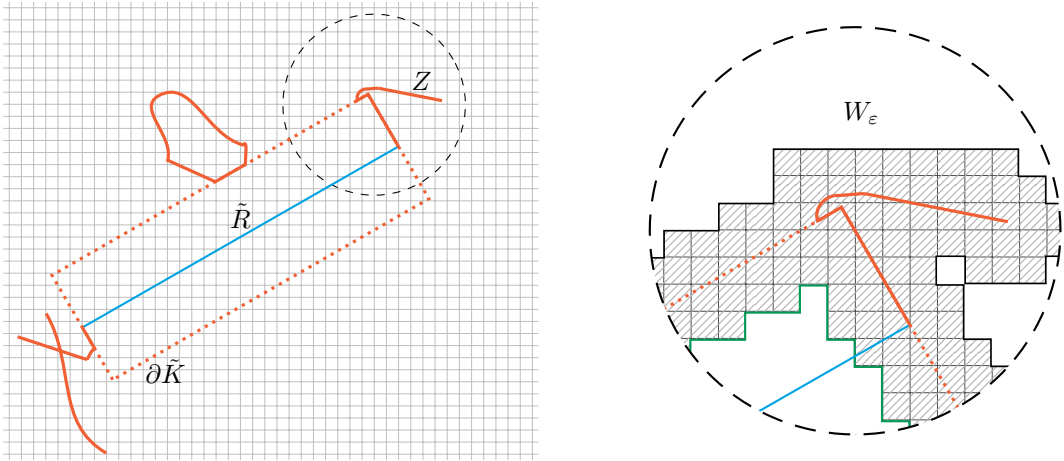
\begin{figure}[!h]
\begin{minipage}[c]{.45\textwidth}
\hspace{1.cm}
\begin{tikzpicture}[rotate=30,scale=0.8]
 \draw[step=.2,gray!50!white,very thin,rotate=-30] (-3.9,-3.7) grid (4.9,3.9);
 \draw[dashed](3,.8) circle (1.5cm);
\draw[thick,Cerulean]  (-3,-.4) |- (-2.5,0)--(2.5,0) -|(3,1)--(2.8,1) to[out=80,in=160] (3,1.1) ;
\draw[very thick,RedOrange]  (-4,.4)--(-3.2,-.5) (3.2,1)--(4,.3) node [anchor = south east] {\color{black}{$Z$}};
\draw[very thick,RedOrange] 	(.1,1) -- (-.4,1.7) to[in=180](0,2.7) to[out=0](.8,1.3) to[in=60,out=20](.67,1)--(.1,1);
\draw[very thick,RedOrange]   (-3.2,-.5) to[out=30,in=188] (-3,-.4) (3,1.1)to[out=-30,in=150] (3.2,1) ;
\draw[very thick,RedOrange]  (-3,-.4) -- (-3,0) (3,0) -|(3,1)--(2.8,1) to[out=80,in=160] (3,1.1) ;
\draw (0,0) node [anchor= south] {$\tilde R$};
\draw[very thick,RedOrange] (-3.7,-2)to[out=120,in=270](-3.4,.5);
\draw[very thick,RedOrange,dotted] (-3,-1) rectangle (3,1);
\draw(-2,-1)node[anchor=north]{\color{black}{$\pt\tilde K$}};
\end{tikzpicture}
\end{minipage}
\hfill
\begin{minipage}[c]{.45\textwidth}
\centering
\begin{tikzpicture}[rotate=30,scale=1.8]

 \begin{scope}
    \clip (3,.8) circle (1.5cm);
     \clip(.9,-5)-- (1.07,1.46) -- ++ (-30:.2) -- ++ (60:.2)-- ++ (-30:.4)-- ++ (60:.2)-- ++ (-30:.4)-- ++ (60:.2)-- ++ (-30:.4)-- ++ (60:.4)-- ++ (-30:1.8)
     -- ++ (60:-.2)-- ++ (-30:.2)
 -- ++ (60:-.2)   -- ++ (-30:.2) -- ++ (60:-.4)   -- ++ (-30:-.2) -- ++ (60:-.2) -- ++ (-30:-.6) 
      -- ++ (60:.2) -- ++ (-30:-.2) -- ++ (60:-.2)-- ++ (-30:.2)-- ++ (60:-.4)-- ++ (-30:.2)-- ++ (60:-.2)-- ++ (-30:.2)-- ++ (60:-.4)-- ++ (-30:.2)--cycle;
 \draw[step=.2,black!70!white,very thin,rotate=-30] (0,0) grid (5,5);
 \filldraw[very thick,pattern=north east lines, pattern color=white!40!gray ](.9,-5)-- (1.07,1.46) -- ++ (-30:.2) -- ++ (60:.2)-- ++ (-30:.4)-- ++ (60:.2)-- ++ (-30:.4)-- ++ (60:.2)-- ++ (-30:.4)-- ++ (60:.4)-- ++ (-30:1.8)-- ++ (60:-.2)-- ++ (-30:.2)
 -- ++ (60:-.2)   -- ++ (-30:.2) -- ++ (60:-.4)   -- ++ (-30:-.2) -- ++ (60:-.2) -- ++ (-30:-.6) 
 -- ++ (60:.2) -- ++ (-30:-.2)   -- ++ (60:-.2)-- ++ (-30:.2)-- ++ (60:-.4)-- ++ (-30:.2)-- ++ (60:-.2)-- ++ (-30:.2)-- ++ (60:-.4)-- ++ (-30:.2)--cycle;
  \filldraw[draw=ForestGreen,thick,fill=white](.9,-5)-- (1.01,.56) -- ++ (-30:.2) -- ++ (60:.2)  -- ++ (-30:.2) -- ++ (60:.2) -- ++ (-30:.4) -- ++ (60:.2)-- ++ (-30:.4) -- ++ (60:.2)  -- ++ (-30:.2) 
  -- ++ (60:-.4) -- ++ (-30:.2) -- ++ (60:-.2) -- ++ (-30:.2) -- ++ (60:-.4) -- ++(-30:.2) -- ++ (60:-1) --cycle;
\draw[thick,Cerulean]  (-3,-.4) |- (-2.5,0)--(2.5,0) -|(3,1)--(2.8,1) to[out=80,in=160] (3,1.1) ;
\draw[very thick,RedOrange]  (3.2,1)--(4,.3)  ;
\draw[very thick,RedOrange]  (3,1.1)to[out=-30,in=150] (3.2,1) ;
\draw[very thick,RedOrange]  (3,0) -|(3,1)--(2.8,1) to[out=80,in=160] (3,1.1) ;
\draw[very thick,RedOrange,dotted] (-3,-1) rectangle (3,1);
  \end{scope}
  \draw (3.4,1.4) node [anchor = south] {\color{black}{$W_\eps$}};
  \draw[thick,dash pattern={on 8pt off 5pt}] (3,.8) circle (1.5cm);
\end{tikzpicture}
\end{minipage}
\label{fig6}\caption{On the left we represent the finer grid for the third projection step. In orange we draw the set $Z:=(\supp \tilde R\bks \Int \tilde K)\cup \pt \tilde K$, in particular the dotted part is $\pt \tilde K$. In the right we show a detail of the drawing with the set $W_\eps$. The set $W_\eps$ is highlighted in striped gray and the set $ \pt W_\eps\cap \tilde K$ in dark green.}
\end{figure}

We continue the above construction by deforming $\tilde R$ in a neighborhood of $\tilde K^c$. In particular, we start with $\tilde R$ and $\tilde V$ satisfying~\eqref{Id+supports-3}--\eqref{TildeTPolyhedral}.\\
Let us introduce $Z:=(\supp \tilde R\bks \Int \tilde K)\cup \pt \tilde K$ and for $\eps>0$, $Z_\eps:=Z+{B_\eps}$ its $\eps$-neighborhood. Since $Z$ is closed, for any finite positive Borel measure $\lambda$, there holds $\lambda(Z)=\lim_{\eps\dw0} \lambda(Z_\eps)$. In particular, from~\eqref{assumption-Tmainlyonkcubes}\eqref{estimM1-2}, 
there exists $0<\eps<\delta$ such that 
\be
\label{NRJonWeps}
\MM_h(\tilde R\restr Z_{2\sqrt{n}\eps})\,< \,  C\eta.
\ee
We fix such $\eps$ and we now  consider the $j$-skeletons $\Q^{(0)},\cdots,\Q^{(n)}$ based on the cube $(0,\eps)^n$. We then introduce the open set 
\[
W_\eps:=\inter\lt(\ov{\cup \{Q\in\Q^{(n)} ,d(Q,Z)<\sqrt{n}\eps\} }\rt),
\]
We have $W_\eps\subset Z_{2\sqrt{n}\eps}$ and by~\eqref{NRJonWeps},
\[
\MM_h(\tilde R\restr W_\eps)\,< \,  C\eta.
\]
For $j\in\{0,\cdots,n\}$, we note 
\[
\what \Q^{(j)}:=\{Q\in  \Q^{(j)} : Q\subset W_\eps\}.
\]
Now let us note $\what R^n:=\tilde R$ and let us introduce the union of cubes
\[
Y_\eps:=\inter\lt(\ov{\{Q\in \Q^{(n)}, \ov{Q}\subset \Int\tilde K\}}\rt).
\]
With this definition $W_\eps\cup Y_\eps$ covers $\tilde K$ and we have 
\[
 \pt W_\eps\cap \tilde K\ \subset Y_\eps\,\subset\, \ov{Y_\eps}\,\subset\,  \Int \tilde K.
 \]
Moreover, from~\eqref{TildeTPolyhedral}, 
\be
\label{TnrestYeps}
\what R^n \restr Y_\eps \in \PP_k(\R^n).
\ee
As in the previous subsection, we also introduce 
\[
\what  \om^{(j)}:=\cup \{Q\in \what  \Q^{(i)}: j+1\leq i\leq n \}=\cup\{\om_Q:Q\in \what  \Q^{(j)}\}, 
\]
and
\[
\what  \Sigma^{(j)}:=\ov{W_\eps} \bks \what \om^{(j)}=\pt W_\eps \cup\lt[\cup \{Q\in \what  \Q^{(i)}:0\leq i \leq j\}\rt].
\]
We have
\[
\what  \Sigma^{(j)}\,\subset\, \Sigma_Q\cap \ov{W_\eps}\qquad\mbox{ for every }Q\in \what  \Q^{(j)}.
\]
We perform the same steps as in Subsection~\ref{Sproof}. Starting with the current $\what R^n=\tilde R$, we apply recursively the local deformation lemma for $Q\in \what \Q^{(n)}$. We obtain the decomposition 
\begin{equation}\label{Id+supports-4}
\begin{gathered}
\what R^n=\what R^{n-1}+ \what U^{n-1} +\pt \what V^{n-1},\qquad \supp U^{n-1}\, \cup\, \supp V^{n-1}\, \subset\, \ov{W_\eps}, \\
\supp {\what R}^{n-1}\subset \lt[\supp{\what R}^{n} \bks \ov{W_\eps}\rt]\cup \what \Sigma^{(n-1)},
\end{gathered}
\end{equation}
with the estimates
\begin{equation}\label{M1-1}
\begin{gathered}
\MM_h(\what R^{n-1}-\what R^n)\,\leq\,c\, \MM_h(\what R^n\restr W_\eps),\qquad \qquad\MM_h(\what V^{n-1})\,\leq\,c\,\delta\, \MM_h(\what R^n\restr W_\eps),\\
\MM_h(\pt \what R^{n-1}-\pt \what R^n)\,\leq\, c\, \MM_h(\pt \what R^n\restr W_\eps),\qquad\qquad
\MM_h(\what U^{n-1})\,\leq\,c\, \delta\, \MM_h(\pt \what R^n\restr W_\eps).
\end{gathered}
\end{equation}
Notice that since $\pt \what R^n\restr W_\eps$ does not necessarily vanish, we have to take into account the component $\what U^n$. On the other hand, by assumption, $\pt\what R^n\in \PP_{k-1}(\R^n)$. Hence, by \eqref{ptTrestQ} of Lemma~\ref{lem1}, $\what U^n$ is a polyhedral current and we have 
\be
\label{ptTnm1}
\pt \what R^{n-1}=\pt \what R^n-\pt \what U^{n-1}\in \PP_{k-1}(\R^n).
\ee
Eventually, by~\eqref{ptTrestQ}~and~\eqref{TrestQ} of Lemma~\ref{lem1}, the property~\eqref{TnrestYeps}  also propagates, we have
\be
\label{Tnm1restYeps}
\what R^{n-1} \restr Y_\eps \in \PP_k(\R^n).
\ee

\medskip
After this first step, we apply the local deformation lemma to $\what R^{n-1}$ with respect to every $Q\in\what\Q^{(n-1)}$ and then with respect to every $Q\in \what\Q^{(n-2)},\cdots$ up to $\what Q^{(k+1)}$. At each step we obtain the properties corresponding to \eqref{Id+supports-4}--\eqref{Tnm1restYeps}. We end up with
\[
\tilde R=\what R^{k} + \what U +\pt \what V,\qquad\mbox{ with }\quad \supp  \what  U\cup \supp  \what V \subset \ov{W_\eps}\cup \tilde K.
\]
Moreover, 
\[
\supp \what R^{k}\subset \lt[\supp{\what R}^{n} \bks \ov{W_\eps}\rt]\cup \what \Sigma^{(k)},
\]
and we have the  estimates
\begin{eqnarray*}
\MM_h(\what R^{k}-\tilde R)\,  \leq\, C\, \MM_h(\tilde R \restr W_\eps),&\qquad&   \MM_h(\what V) \, \leq\, C\,\delta\, \MM_h(\tilde R \restr W_\eps),\\
\MM_h(\pt \what R^{k}-\pt \tilde R)\, \leq\, C\, \MM_h(\pt \tilde R\restr W_\eps),&\qquad&  C\, \MM_h(\pt \tilde R\restr W_\eps), \, \leq\, C\, \delta\,\MM_h(\pt \tilde R\restr W_\eps).
\end{eqnarray*}
By~\eqref{ptTrestQ} 
of Lemma~\ref{lem1}, 
\[
\what U\in \PP_k(\R^n).
\]
By construction, $\what R^k \restr W_\eps=\what R^k\restr (W_\eps\cap \what \Sigma^{(k)})$ and $\pt \what R^k\in \PP_{k-1}(\R^n) $, so that by Lemma~\ref{lemConstancy},
\[
\what R^k\restr W_\eps\in \PP_k(\R^n).
\]
Now, $\what R^k\restr W_\eps\cup Y_\eps)^c=0$, so we have to check that $\what R^k\restr Y_\eps\in \PP_k(\R^n)$. From~\eqref{TnrestYeps} we have $\what R^n\restr Y_n\in \PP_k(\R^n)$ and this property propagates by~\eqref{TrestQ}. We conclude that 
\[
\what R^k\in \PP_k(\R^n).
\]
Finally, we set 
\[
P:=\what R^k - \what U,\qquad V:=\tilde V + \what V. 
\]
We have $P\in \PP_k(\R^n)$ and the currents $P$ and $V$ satisfy the estimates stated in Theorem~\ref{thm1}. This concludes the proof of the theorem.

\section{The case $\MM_h\lesssim\MM$.}
\label{Sbeta00}
In this last part, we consider the case $\beta:=\sup_{\theta>0}h(\theta)/\theta<\oo$. Before proving Theorem~\ref{thm2}, we start with a description of the decomposition $T=R+T'$ introduced in~\eqref{wMMh} for the definition of $\widehat{\MM}_h$.
\subsection{Decomposition of finite mass flat chains into rectifiable and diffuse parts}
Assume that $T\in \FF_k(\R^n)$ has finite mass. The upper $k$-density of $T$ at a point $x\in \R^n$ is defined  as
\[
\Theta^*_k(T)(x)=\liminf_r \dfrac{\MM(T\restr B(x,r)}{r^k}.
\] 
Then, for $\eps\geq0$, we define the Borel set,
\[
X^\eps:=\lt\{x\in \R^n:\Theta^*_k(T)(x)>\eps\rt\} .
\] 
Since $\MM(T)<\oo$, the restriction of $T$ to any Borel set $X$   is a well defined flat chain and have all the desired properties --- see~\cite[Sec. 4]{Fleming66}. Noting $T\restr X\in \FF_k(\R^n)$ this restriction, we have in particular $\MM(T)=\MM(T\restr X)+\MM(T\restr X^c)$ (beware that Fleming uses the notation $T\cap X$ for $T\restr X$). We note 
\[
R^\eps:=T\restr X^\eps,\qquad\mbox{for $\eps\geq0$}.
\] 
We have $\MM(R^\eps)\leq \MM(T)$ and by a classical covering argument, for $\eps>0$, there holds $\H^k(X^\eps)\leq C \MM(T)/\eps$ where the constant $C$ only depends on $n$.  Consequently, $R^\eps$ has finite size and finite mass and by the rectifiability theorem of White~\cite[Proposition 8.2]{White2}, $R^\eps$ is rectifiable. Taking the limit $\eps\dw 0$, we see that $R:=R^0\in \RR_k(\R^n)$. Eventually, we set $T':=T-R$ and by construction, 
\be\label{Thetak}
\Theta^*_k(T')\equiv 0\mbox{ in }\R^n.
\ee
This is what we mean by $T'$ is a ``diffuse'' flat chain. The decomposition $T=R+T'$ is uniquely characterized by the three properties $R\in \RR_k(\R^n)$, $T'\in \FF_k(\R^n)$ satisfies~\eqref{Thetak} and $\MM(T)=\MM(R)+\MM(T')$.

\subsection{Proof of Theorem~\ref{thm2}}
We assume that $h:\R\to\R_+$ satisfies~\eqref{condf}, \eqref{condf2} and~\eqref{beta}. We set $T=R+T'$ with $T'\in \FF_k(\R^n), R\in \RR_k(\R^n)$ and  $\MM(R)+\MM(T')<\oo$. Let $\eps>0$.\\ 
First, we apply Proposition~\ref{Prop 2.6} to the rectifiable current $R$: we have
\[
R=P_1+U_1+\pt V_1,\quad\mbox{with}\quad  \MM_h(P_1)\ <\ \MM_h(R)+\eta\quad\mbox{and}\quad \MM(U_1)+\MM(V_1)\ <\ \eps.
\]
Next, by~\cite[Theorem~4.1.23]{federer}, there exist $P\in \PP_k(\R^n)$, $U_1'\in \FF_k(\R^n)$ and $V_1'\in \FF_{k+1}(\R^n)$ such that 
\[
T'=P_1'+U_1'+\pt V_1',\qquad \MM(P_1')\ <\ \MM(T')+\eps\quad\mbox{and}\quad \MM(U_1')+\MM(V_1')\ <\ \eps.
\] 
In fact the result stated in~\cite{federer} assumes that $T$ is compactly supported but the general case can be recovered easily.\footnote{ Consider a smooth contraction $u$ of $\R^n$ with range in $B_2$ and such that $u\equiv \Id$ in $B_1$ and define $u_r(x):=ru(x/r)$. For $r$ large enough  $\MM(u-u_r\pf T)<\eta$ and we can apply the result to the compactly supported flat chain $u_r\pf T$}
Setting $P=P_1+P_1'$,  $U=U_1+U_1'$ and $V=V_1+V_1'$, we have $T=P+U+\pt V$ with the estimates
\begin{eqnarray*}
&\MM(U)+\MM(V)\leq\MM(U_1)+\MM(V_1)+\MM(U_1')+\MM(V_1')<2\eps&\\
&\MM_h(P)\leq\MM_h(P_1)+\MM_h(P_1')\leq\MM_h(P_1)+\beta\MM(P_1')<\MM_h(R)+\beta\MM(T')+(1+\beta)\eps.&
\end{eqnarray*}
Choosing $\eps$ such that $(2+\beta)\eps<\eta$, the first part of the theorem is proved.\\
Eventually, if we assume that $\pt T$ is polyhedral, we have $\pt U=\pt T-\pt P\in P_{k=1}(\R^n)$. In particular, $U$ is a normal current and we can apply the deformation theorem of Federer and Fleming to $U$ (see~\cite{FedFlem60},~\cite[4.2.9]{federer} or \cite{KP2008}) to get the decomposition $U=P_2+U_2+\pt V_2$ with 
\[
\MM(P_2)+\MM(U_2)+\MM(V_2)<\MM(U)+\eps.
\]
Moreover since $\pt U$ is polyhedral, $U_2$ is polyhedral, so that setting $\tilde P=P+P_2+U_2\in \PP_k(\R^n)$ and $\tilde V=V+V_2$, we have the decomposition $T=\tilde P+ \pt\tilde V$ with 
\[
\MM(\tilde V)\leq \MM(V)+\MM(V_2)<3\eps,\qquad \MM_h(\tilde P)\leq\MM_h(P) +\beta \MM(P_2)< \MM_h(R)+\beta\MM(T')+(1+2\beta)\eps.
\]
Choosing $\eps>0$ small enough, we obtain the desired estimates.

\subsection*{Acknowledments}
B. Merlet is partially supported by the INRIA team RAPSODI and the Labex CEMPI (ANR-11-LABX-0007-01).

\bibliographystyle{plain}
\bibliography{DefTheorem}
\end{document}